\documentclass[12pt]{amsart}
\usepackage{geometry}
\usepackage{url}
\usepackage{graphicx}
\usepackage[usenames]{color}
\usepackage{amsmath}
\usepackage{amssymb}
\usepackage{amsfonts}
\usepackage[english]{babel}
\usepackage{enumerate}
\usepackage{mathtools}
\usepackage{mathrsfs}
\usepackage[T1]{fontenc}
\usepackage{amsthm}
\usepackage{epstopdf}
\usepackage{hyperref}
\usepackage{tikz-cd}

\newtheorem{thm}{Theorem}[section]
\newtheorem{definition}[thm]{Definition}
\newtheorem{prop}[thm]{Proposition}
\newtheorem{cor}[thm]{Corollary}
\newtheorem{rk}[thm]{Remark}

\newtheorem{lemma}[thm]{Lemma}
\newtheorem{notation}[thm]{Notation}

\DeclareMathOperator{\image}{Im}

\DeclareMathOperator{\interior}{int}

\DeclareMathOperator{\sect}{Sect}
\DeclareMathOperator{\codim}{codim}

\DeclareMathOperator{\Hess}{Hess}
\DeclareMathOperator{\Diff}{Diff}

\begin{document}

\title{Bumpy metrics theorem for geodesic nets}
\author{Bruno Staffa}
\begin{abstract}
Stationary geodesic networks are the analogs of closed geodesics whose domain is a graph instead of a circle. We prove that for a Baire-generic Riemannian metric on a smooth manifold $M$, all connected embedded stationary geodesic nets are non-degenerate.
\end{abstract}

\maketitle

\section*{Introduction}
Let $M$ be an $n$-dimensional smooth manifold and let $\Gamma$ be a weighted multigraph. Let $\mathscr{E}$ be the set of edges of $\Gamma$, $\mathscr{V}$ the set of vertices and for each $E\in\mathscr{E}$ let $n(E)\in\mathbb{N}$ be its multiplicity. Consider the spaces
\begin{align*}
    \mathcal{M}^{k} & =\{g:g\text{ is a }C^{k}\text{ Riemannian metric on }M\},\\
    \Omega(\Gamma,M) & =\{f:\Gamma\to M:f\text{ is continuous and }f|_{E} \text{ is a } C^{2}\text{ immersion }\forall E\in\mathscr{E}\}.
\end{align*}
We say that $f_{0}\in\Omega(\Gamma,M)$ is a stationary geodesic network with respect to a metric $g\in\mathcal{M}^{k}$ if it is a critical point of the length functional $l_{g}:\Omega(\Gamma,M)\to\mathbb{R}$ defined as
\begin{equation*}
    l_{g}(f)=\int_{\Gamma}\sqrt{g_{f(t)}(\dot{f}(t),\dot{f}(t))}dt=\sum_{E\in\mathscr{E}}n(E)\int_{E}\sqrt{g_{f(t)}(\dot{f}(t),\dot{f}(t))}dt.
\end{equation*}
In other words, $f_{0}$ is stationary with respect to $g$ if for every one parameter family $f:(-\varepsilon,\varepsilon)\to\Omega(\Gamma,M)$ with $f(0,\cdot)=f_{0}$ we have
\begin{equation*}
    \frac{d}{ds}\bigg|_{s=0}l_{g}(f_{s})=0.
\end{equation*}
Given a stationary geodesic network $f_{0}$ with respect to $g\in\mathcal{M}^{k}$, we can consider the Hessian of $l_{g}$ at $f_{0}$:
\begin{equation*}
    \Hess_{f_{0}} l_{g}(X,Y)=\frac{\partial^{2}}{\partial x\partial s}\bigg|_{(0,0)}l_{g}(f(x,s))
\end{equation*}
where $X,Y$ are $C^{2}$ vector fields along $f_{0}$ and $f:(-\varepsilon,\varepsilon)^{2}\to\Omega(\Gamma,M)$ is a two parameter family with $f_{00}=f_{0}$ verifying $\frac{\partial f}{\partial s}(0,0,t)=X(t)$ and $\frac{\partial f}{\partial x}(0,0,t)=Y(t)$. A vector field $J$ along $f_{0}$ is said to be Jacobi if $\Hess_{f_{0}}l_{g}(X,J)=0$ for every vector field $X$ along $f_{0}$. It is easy to check that every parallel vector field along $f_{0}$ (i.e. any continuous vector field $J$ which is parallel and $C^{2}$ when restricted to each edge $E$ of $\Gamma$) is Jacobi. Therefore we say that $f_{0}$ is a nondegenerate stationary geodesic network with respect to $g$ if every Jacobi field along $f_{0}$ is parallel (notice that this is analogous to the notion of nondegeneracy for minimal submanifolds). We say that a metric $g\in\mathcal{M}^{k}$ is bumpy if every embedded stationary geodesic network with respect to $g$  whose domain is a good weighted multigraph $\Gamma$ is nondegenerate (see Section \ref{setup} for the definition of good weighted multigraph).

Our goal is to prove that for each $k\in\mathbb{
N}_{\geq 3}\cup\{\infty\}$ the set of bumpy metrics is ``big'', in the sense it is Baire-generic in $\mathcal{M}^{k}$. To achieve that, we will study the space of stationary geodesic networks for varying Riemannian metrics on $M$. We will follow the ideas of \cite{White}, where this problem is studied for embedded minimal submanifolds, and adapt the arguments developed there to our setting. The main difference with the minimal submanifold problem is that our objects (stationary geodesic networks) are not everywhere smooth. Therefore, when we want to model a neighborhood of some $f_{0}\in\Omega(\Gamma,M)$ we have to consider two degrees of freedom that determine a nearby $f\in\Omega(\Gamma,M)$: one is related to the image of the vertices and the other with the map along the edges. In order to have an injective parametrization of these geometric objects, we will mod out by reparametrizations and work with the quotient space $\hat{\Omega}(\Gamma,M)=\Omega(\Gamma,M)/\sim$ where $f\sim g$ if and only if there exists a homeomorphism $\tau:\Gamma\to\Gamma$ such that $\tau$ fixes the vertices of the graph, $\tau(E)=E$ for all $E\in\mathscr{E}$ and $\tau|_{E}:E\to E$ is a $C^{2}$ diffeomorphism for all $E\in\mathscr{E}$. In Section \ref{paths} of this paper we study $\hat{\Omega}([0,1],M)$ (i.e. the space of immersed paths on $M$ under reparametrization) and show that any $[f]$ close to $[f_{0}]\in\hat{\Omega}([0,1],M)$ can be obtained by a composition of a horizontal displacement (moving the vertices along an extension of the smooth curve $f_{0}:[0,1]\to M$) and a normal one (moving in the direction of a normal vector field along $f_{0}$ with respect to a background metric $\gamma_{0}$). Therefore $\hat{\Omega}([0,1],M)$ is modeled by the Banach space $\mathbb{R}^{2}\times\sect(N_{f_{0}})$ where $\sect(N_{f_{0}})$ denotes the space of $C^{2}$ sections of the normal bundle $N_{f_{0}}$ along $f_{0}:[0,1]\to M$ with respect to the background metric $\gamma_{0}$. Here we see the difference with the closed submanifold case analysed in \cite{White}, where the space of normal vector fields along a minimal submanifold $f_{0}:N\to M$ models a neighborhood of $[f_{0}]$; while for paths we have an additional $\mathbb{R}^{2}$ factor because there is an extra degree of freedom for each vertex. Those extra degrees of freedom will also be present in the spaces $\hat{\Omega}(\Gamma,M)$ we are interested in, as it is shown in Section \ref{lengthsgn} where a $C^{0}$ (but not differentiable) Banach manifold structure is given to those spaces.

Once we have such structure for $\hat{\Omega}(\Gamma,M)$, it is possible to derive the first and second variation formulas for the length functional in local coordinates. We obtain expressions analogous to those derived in \cite{White} but with additional terms corresponding to the vertices. This allows us to understand the space
\begin{equation*}
    \mathcal{S}^{k}_{0}(\Gamma)=\{(g,f)\in\mathcal{M}^{k}\times\hat{\Omega}(\Gamma,M):f\text{ is stationary with respect to }g\}
\end{equation*}
locally as the set of zeros of a mean curvature map $H:\mathcal{M}^{k}\times C_{0}\to\mathcal{Y}$ where $C_{0}$ is a Banach manifold which is the image of $\hat{\Omega}(\Gamma,M)$ under a chart, $\mathcal{Y}$ is a suitable Banach space that is defined in Section \ref{lengthsgn} and $H$ is a $C^{k-2}$ map between Banach manifolds. We use \cite[Theorem~1.2]{White} to give a Banach manifold structure to an open subset $\mathcal{S}^{k}(\Gamma)\subseteq\mathcal{S}^{k}_{0}(\Gamma)$. In order to do that, we prove in Section 4 that $D_{2}H$ is Fredholm of index $0$. Additionally, to satisfy condition (C) of \cite[Theorem~1.2]{White}, we restrict our attention to good weighted multigraphs $\Gamma$ and to embedded $\Gamma$-nets as defined in Section \ref{setup}. We denote
\begin{equation*}
    \Omega^{emb}(\Gamma,M)=\{f\in\Omega(\Gamma,M):f\text{ is embedded}\}.
\end{equation*}
As we show in Section \ref{lengthsgn}, this technical condition rules out the possibility of having parallel Jacobi fields along $[f]\in\hat{\Omega}(\Gamma,M)$ and allows us to give a Banach manifold structure to
\begin{equation*}
    \mathcal{S}^{k}(\Gamma)=\{(g,f)\in\mathcal{M}^{k}\times\hat{\Omega}^{emb}(\Gamma,M):f\text{ is stationary with respect to }g\}\subseteq\mathcal{S}^{k}_{0}(\Gamma).
\end{equation*}

\begin{rk}
It is proved in \cite[Lemma~2.5]{Liokumovich} that given a stationary geodesic net $f:\Gamma\to M$ (with respect to a metric $g$), there exist $\{f_{i}:\Gamma_{i}\to M\}$ where each $\Gamma_{i}$ is a good weighted multigraph and each $f_{i}:\Gamma_{i}\to M$ is an embedded stationary geodesic net such that their union has the same image and multiplicity at every point as $f$. Hence we do not loose much generality by restricting our attention to good multigraphs and embedded stationary geodesic nets.
\end{rk}


Having the previous considerations in mind and applying \cite[Theorem~1.2]{White} as mentioned before, we prove in Section \ref{banachmanstr}  that $\mathcal{S}^{k}(\Gamma)$
is a $C^{k-2}$ Banach manifold and that the projection $\Pi:\mathcal{S}^{k}(\Gamma)\to\mathcal{M}^{k}$, $(g,f)\mapsto g$ is Fredholm of index $0$. This can be summarized in the following structure theorem.

\begin{thm}[Structure theorem for geodesic nets]\label{structurethm}
Let $\Gamma$ be a good weighted multigraph and $k\in\mathbb{N}_{\geq 3}$. Then
\begin{enumerate}
    \item The space
    \begin{equation*}
        \mathcal{S}^{k}(\Gamma)=\{(g,f)\in\mathcal{M}^{k}\times\hat{\Omega}^{emb}(\Gamma,M):f\text{ is stationary with respect to }g\}
    \end{equation*}
    has a $C^{k-2}$ Banach manifold structure.
\item  The projection map $\Pi:\mathcal{S}^{k}(\Gamma)\to\mathcal{M}^{k}$ onto the first coordinate is Fredholm of index $0$.
\item Given $(g,f)\in\mathcal{S}^{k}(\Gamma)$, $f$ is nondegenerate with respect to $g$ if and only if $D\Pi_{(g,f)}:T_{(g,f)}\mathcal{S}^{k}(\Gamma)\to T_{g}\mathcal{M}^{k}$ is an isomorphism.
\end{enumerate}

\end{thm}

The previous theorem together with Smale's version of Sard's theorem for Banach spaces from \cite{Smale} implies

\begin{thm}[Bumpy metrics theorem for stationary geodesic nets]\label{bumpythm}
Given $k\in\mathbb{N}_{\geq 3}\cup\{\infty\}$ the subset $\mathcal{N}^{k}\subseteq\mathcal{M}^{k}$ of bumpy metrics is generic in the Baire sense.
\end{thm}

To be precise, Theorem \ref{structurethm} and Theorem \ref{bumpythm} for $k\in\mathbb{N}_{\geq 3}$ are proved in Section 5 using the fact that $C^{k}$ spaces have a Banach manifold structure. Although the same reasoning does not hold immediately for $C^{\infty}$ spaces because they only have Frechet structures, in Section \ref{Cinfty} we extend Theorem \ref{bumpythm} to $C^{\infty}$ metrics.

\begin{rk}
Observe that our result does not provide nondegeneracy for not embedded stationary geodesic networks $f:\Gamma\to M$. In particular, we do not rule out the possibility of having a sequence of non-smooth stationary geodesic nets $f_{n}:\Gamma\to M$ converging to a stationary geodesic net $f_{0}:\Gamma\to M$ which represents a closed geodesic loop with certain multiplicity (for example, a sequence of stationary figure eights which converges to a simple closed geodesic with multiplicity $2$).
\end{rk}

Theorem \ref{structurethm} allowed to prove that for a generic metric in a closed manifold $M$, the union of all stationary geodesic nets forms a dense subset of $M$ (see the work \cite{Liokumovich}). More recently, Theorem \ref{structurethm} was used in \cite{LiSta} to prove that for a generic Riemannian metric $g$ in a closed $2$-manifold (respectively $3$-manifold), there exists a sequence of closed geodesics (respectively of embedded stationary geodesic networks) which is equidistributed in $(M,g)$.

\begin{rk}
We recently learnt that Otis Chodosh and Christos Mantoulidis have independently proved a different Bumpy Metrics Theorem for stationary geodesic networks in 2-manifolds as part of their work \cite{Chodosh}, where they proved several remarkable results including the computation of the Weyl law constant for surfaces and the fact that min-max stationary geodesic networks on surfaces are unions of immersed closed geodesics.
\end{rk}

\vspace{0.2in}

\textbf{Acknowledgements.}
I am grateful to Yevgeny Liokumovich for suggesting this problem and for his valuable guidance. I also want to thank Otis Chodosh and Christos Mantoulidis for their valuable comments and suggestions. The author was partially supported by NSERC Discovery grant.

\section{Set up}\label{setup}
\begin{definition}
A weighted multigraph is a graph $\Gamma=(\mathscr{E},\mathscr{V},\{\pi_{E}\}_{E\in\mathscr{E}},\{n(E)\}_{E\in\mathscr{E}})$ consisting of:
\begin{enumerate}
    \item A set of edges $\mathscr{E}$. For each $E\in\mathscr{E}$, we fix an homeomorphism $E\cong[0,1]$.
    \item A set of vertices $\mathscr{V}$.
    \item For each $E\in\mathscr{E}$, a map $\pi_{E}:\{0,1\}\to\mathscr{V}$ which sends each of the boundary points of the edge $E$ (identified with $0$ and $1$) to their corresponding vertex $v$.
    \item A multiplicity $n(E)\in\mathbb{N}$ assigned to each edge $E\in\mathscr{E}$.
\end{enumerate}
We will also denote by $\Gamma$ the one-dimensional simplicial complex $\mathscr{E}\times[0,1]/\sim$ where $(E,s)\sim (E',s')$ if and only if $s,s'\in\{0,1\}$ and $\pi_{E}(s)=\pi_{E'}(s')$.

\end{definition}

\begin{definition}\label{incoming edge}
    Let $\Gamma$ be a weighted multigraph. Given a vertex $v\in\mathscr{V}$, an incoming edge at $v$ is a pair $(E,i)\in\mathscr{E}\times\{0,1\}$ such that $\pi_{E}(i)=v$. We will assume that every vertex of the weighted multigraphs $\Gamma$ we work with has at least two different incoming edges.
\end{definition}

\begin{rk}
    Notice that if we consider the simplicial complex associated to $\Gamma$, each loop edge at $v$ appears two times as an incoming edge at $v$ (as $(E,0)$ and as $(E,1)$) and all the other edges appear exactly once (either as $(E,0)$ or as $(E,1)$).
\end{rk}

\begin{definition}
A weighted multigraph $\Gamma$ is good* if the underlying one-dimensional simplicial complex is connected and each vertex $v \in \mathscr{V}$ has at least three different incoming edges. A weighted multigraph is good if either it is good* or it is a simple loop with multiplicity.
\end{definition}

\begin{definition}
A $\Gamma$-net $f$ on $M$ is a continuous map $f:\Gamma\to M$ which is a $C^{2}$ immersion when restricted to the edges of $\Gamma$. The previous means that for each $E\in\mathscr{E}$ the map
\begin{center}
\begin{tikzcd}
f_{E}:[0,1]\rar{\iota_{E}} &\mathscr{E}\times[0,1] \rar{q} & \Gamma \rar{f} & M
\end{tikzcd}
\end{center}
is a $C^{2}$ immersion (here $\iota_{E}(t)=(E,t)$ and $q$ is the quotient map $q:\mathscr{E}\times[0,1]\to \Gamma=\mathscr{E}\times[0,1]/\sim$). We think of $f_{E}$ as the restriction of $f$ to the edge $E$ and sometimes regard its domain as $E$ under the identification $E\cong[0,1]$. We denote $\Omega(\Gamma,M)$ the space of $\Gamma$-nets on $M$.
\end{definition}

\begin{definition}
Given $f\in\Omega(\Gamma,M)$ and $k\geq 0$, we denote $\mathfrak{X}^{k}(f)$ the space of continuous vector fields along $f$ which are of class $C^{k}$ along each edge of $\Gamma$ (observe that $\mathfrak{X}^{k}(f_{0})$ is always well defined for $k\leq 2$ and could be defined for bigger values of $k$ provided the restrictions of $f$ to the edges have enough regularity).
\end{definition}


\begin{notation}
Given $f\in\Omega(\Gamma,M)$, $X\in\mathfrak{X}^{2}(f)$, $E\in\mathscr{E}$ and $t\in E$ we will denote $\dot{X}_{E}(t)$ the covariant derivative of the vector field $X$ along $f_{E}$ at $t$ (with respect to a certain Riemannian metric to be specified). Notice that when $t$ is a vertex of $\Gamma$ this definition depends on $E$. We will omit the subscript $E$ when it is implicit which edge are we differentiating along.
\end{notation}

\begin{definition}
We say that a $\Gamma$-net $f$ is embedded if the map $f:\Gamma\to M$ is injective (notice that by the compactness of $\Gamma$ this is equivalent to say that the map $f:\Gamma\to M$ is a homeomorphism onto its image). We denote \begin{equation*}
    \Omega^{emb}(\Gamma,M)=\{f\in\Omega(\Gamma,M):f\text{ is embedded}\}.
\end{equation*}
\end{definition}
 
The spaces $\Omega(\Gamma,M)$ and $\Omega^{emb}(\Gamma,M)$ have natural Banach manifold structures with the $C^{2}$ topology (both are open subspaces of the space $C^{2}(\Gamma,M)$ of continuous maps $f:\Gamma\to M$ which are of class $C^{2}$ along each edge). Let $\mathcal{M}^{k}$ be the space of $C^{k}$ Riemannian metrics on $M$. In the following we will omit the superscript $k$ for simplicity, assuming it is fixed. Given $g\in\mathcal{M}$ and $f\in\Omega(\Gamma,M)$, we define the $g$-length of $f$ by
\begin{equation*}
    l_{g}(f)=\int_{\Gamma}\sqrt{g_{f(t)}(\Dot{f}(t),\Dot{f}(t))}dt
\end{equation*}
where given a measurable function $h:\Gamma\to\mathbb{R}$ which is integrable along each edge $E\in\mathscr{E}$, we define
\begin{equation*}
    \int_{\Gamma}h(t)dt=\sum_{E\in\mathscr{E}}n(E)\int_{E}h(t)dt.
\end{equation*}

\begin{definition}
A $\Gamma$-net $f\in\Omega(\Gamma,M)$ is a stationary geodesic network with respect to the metric $g\in\mathcal{M}$ if it is a critical point of the length functional $l_{g}:\Omega(\Gamma,M)\to\mathbb{R}$.
\end{definition}

In order to give a more precise description of this condition, and to define what it means for a stationary geodesic network to be nondegenerate, we derive the first and second variation formulas for the length functional on $\Omega(\Gamma,M)$.

Let $f:(-\varepsilon,\varepsilon)\times\Gamma\to M$ be a one parameter family of $\Gamma$-nets through $f_{0}=f(0,\cdot)$ and let $X(t)=\frac{\partial f}{\partial s}(0,t)$ be the corresponding variational vector field along $f_{0}$. Then
\begin{equation}\label{eq1}
    \frac{d}{ds}\bigg|_{s=0}l_{g}(f_{s})=\int_{\Gamma}\frac{g_{f_{0}(t)}(\dot{X}(t),\dot{f}_{0}(t))}{\sqrt{g_{f_{0}(t)}(\dot{f}_{0}(t),\dot{f}_{0}(t))}}dt.
\end{equation}
To simplify the computation we will assume that each edge of $f_{0}$ is parametrized with constant speed (we don't loose generality by doing so because every $\Gamma$-net can be reparametrized with constant speed in a unique way), being $\sqrt{g_{f_{0}(t)}(\dot{f}_{0}(t),\dot{f}_{0}(t))}=l_{g}(f_{0}(E))$ for all $t\in E$. Denoting $l_{g}(f_{0}(E))=l(E)$ for simplicity, we get

\begin{equation*}
    \frac{d}{ds}\bigg|_{s=0}l_{g}(f_{s})=\sum_{E\in\mathscr{E}}\frac{n(E)}{l(E)}\int_{E}g_{f_{0}(t)}(\dot{X}(t),\dot{f}_{0}(t))dt.
\end{equation*}
Integrating by parts we obtain

\begin{equation*}
    \frac{d}{ds}\bigg|_{s=0}l_{g}(f_{s})=-\sum_{E\in\mathscr{E}}\frac{n(E)}{l(E)}\int_{E}g_{f_{0}(t)}(X(t),\ddot{f}_{0}(t))dt+\sum_{v\in\mathscr{V}}g_{f_{0}(v)}(X(v),V(f_{0})(v))
\end{equation*}
where
\begin{equation*}
    V(f_{0})(v):=\sum_{(E,i):\pi_{E}(i)=v}(-1)^{i+1}n(E)\frac{\dot{f}_{0,E}(i)}{|\dot{f}_{0,E}(i)|}
\end{equation*}
and $f_{0,E}=(f_{0})_{E}$.

From the previous computation, we see that a constant speed parametrized $\Gamma$-net $f_{0}$ is stationary with respect to $l_{g}$ if and only if:
\begin{enumerate}
    \item $\ddot{f}_{0}(t)=0$ along each edge $E\in\mathscr{E}$ (i.e. the edges of $\Gamma$ are mapped to geodesic segments).
    \item $V(f_{0})(v)=0$ for all $v\in\mathscr{V}$. This means that the sum with multiplicity of the inward unit tangent vectors to the edges concurring at each vertex $v$ must be $0$.
\end{enumerate}

Now assume $f_{0}$ is parametrized with constant speed and stationary. We want to define a continuous bilinear map $\Hess_{f_{0}} l_{g}:\mathfrak{X}^{2}(f_{0})\times\mathfrak{X}^{2}(f_{0})\to\mathbb{R}$ which will be the Hessian of $l_{g}$ at the critical point $f_{0}$ in the following way.  Consider a two parameter variation $f:(-\varepsilon,\varepsilon)^{2}\times\Gamma\to M$ with $f(0,0)=f_{0}$. Let $X(t)=\frac{\partial f}{\partial s}(0,0,t)$ and $Y(t)=\frac{\partial f}{\partial x}(0,0,t)$. We set $\Hess_{f_{0}}(X,Y)=\frac{\partial^{2}}{\partial x\partial s}\bigg|_{(0,0)}l_{g}(f(x,s))$. Next we will compute that expression and show that it is well defined (i.e. that it is independent of the two parameter family $f(x,s)$). From (\ref{eq1}),
\begin{align*}
    \Hess_{f_{0}}l_{g}(X,Y)= & \frac{d}{dx}\bigg|_{x=0}\sum_{E\in\mathscr{E}}n(E)\int_{E}g_{f_{x0}(t)}(\frac{D}{dt}\frac{\partial f}{\partial s}(x,0,t),\frac{\frac{\partial f}{\partial t}(x,0,t)}{|\frac{\partial f}{\partial t}(x,0,t)|})dt\\
    = & \sum_{E\in\mathscr{E}}n(E)\int_{E}g_{f_{0}(t)}(\frac{D}{dx}\frac{D}{dt}\frac{\partial f}{\partial s}(x,0,t)\bigg|_{(0,0,t)},\frac{\dot{f}_{0}(t)}{|\dot{f}_{0}(t)|})dt\\
    & +\sum_{E\in\mathscr{E}}n(E)\int_{E}g_{f_{0}(t)}(\dot{X}(t),\frac{D}{dx}\frac{\frac{\partial f}{\partial t}(x,0,t)}{|\frac{\partial f}{\partial t}(x,0,t)|}\bigg|_{(0,0,t)})dt.
\end{align*}
Computing each sum separately we get

\begin{multline}\label{SVF1}
    \Hess_{f_{0}}l_{g}(X,Y) =  \sum_{E\in\mathscr{E}}\frac{n(E)}{l(E)}\bigg[\int_{E}g(\dot{X}(t),\dot{Y}(t))-g(\dot{Y}(t),\frac{\dot{f}_{0}(t)}{|\dot{f}_{0}(t)|})g(\dot{X}(t),\frac{\dot{f}_{0}(t)}{|\dot{f}_{0}(t)|})\\
     -g(R(\dot{f}_{0}(t),Y(t))\dot{f}_{0}(t),X(t))dt\bigg]+n(E)g(\frac{D}{dx}\frac{\partial f}{\partial s}|_{(0,0,\pi_{E}(i))},\frac{\dot{f}_{0,E}(i)}{|\dot{f}_{0,E}(i)|})\bigg |_{0}^{1}.  
\end{multline}
Observe that
\begin{align*}
    & \sum_{E\in\mathscr{E}}n(E)g(\frac{D}{dx}\frac{\partial f}{\partial s}|_{(0,0,\pi_{E}(i))},\frac{\dot{f}_{0,E}(i)}{|\dot{f}_{0,E}(i)|})\bigg |_{0}^{1}& \notag\\  &=\sum_{v\in\mathscr{V}}\sum_{(E,i):\pi_{E}(i)=v}(-1)^{i+1}n(E)g(\frac{D}{dx}\frac{\partial f}{\partial s}|_{(0,0,v)},\frac{\dot{f}_{0,E}(i)}{|\dot{f}_{0,E}(i)|})\\
    &=\sum_{v\in\mathscr{V}}g(\frac{D}{dx}\frac{\partial f}{\partial s}\bigg |_{(0,0,v)},\sum_{(E,i):\pi_{E}(i)=v}(-1)^{i+1}n(E)\frac{\dot{f}_{0,E}(i)}{|\dot{f}_{0,E}(i)|})\\
    &=0
\end{align*}
because $V(f_{0})(v)=0$  for all $v\in\mathscr{V}$. Using this and integrating by parts the first two terms of (\ref{SVF1}) we get

\begin{multline*}
    \Hess_{f_{0}}l_{g}(X,Y) =\\  
     \sum_{E\in\mathscr{E}}\frac{n(E)}{l(E)}\bigg[\int_{E}g(-\ddot{Y}(t)-R(\dot{f}_{0}(t),Y(t))\dot{f}_{0}(t)+g(\ddot{Y}(t),\frac{\dot{f}_{0}(t)}{|\dot{f}_{0}(t)|})\frac{\dot{f}_{0}(t)}{|\dot{f}_{0}(t)|},X(t))dt\\
    +g(\dot{Y}_{E}(i)-g(\dot{Y}_{E}(i),\frac{\dot{f}_{0,E}(i)}{|\dot{f}_{0,E}(i)|})\frac{\dot{f}_{0,E}(i)}{|\dot{f}_{0,E}(i)|},X(\pi_{E}(i))) \bigg|_{0}^{1}
    \bigg].
\end{multline*}
Therefore we can define a second order differential operator $A_{E}$ along the edge $E$ as

\begin{align*}
    A_{E}(Y) & =\frac{n(E)}{l(E)}\bigg[-\ddot{Y}(t)-R(\dot{f}_{0}(t),Y(t))\dot{f}_{0}(t)+g(\ddot{Y}(t),\frac{\dot{f}_{0}(t)}{|\dot{f}_{0}(t)|})\frac{\dot{f_{0}}(t)}{|\dot{f}_{0}(t)|}\bigg]\\
    & =-\frac{n(E)}{l(E)}\bigg[\ddot{Y}^{\perp}+R(\dot{f}_{0}(t),Y(t)^{\perp}),\dot{f}_{0}(t)\bigg]
\end{align*}
and an operator $B_{v}:\mathfrak{X}^{2}(f_{0})\to T_{f(v)}M$ at each vertex $v\in\mathscr{V}$ as

\begin{align*}
    B_{v}(Y) & =\sum_{(E,i):\pi_{E}(i)=v}(-1)^{i+1}\frac{n(E)}{l(E)}\bigg(\dot{Y}_{E}(i)-g(\dot{Y}_{E}(i),\frac{\dot{f}_{0,E}(i)}{|\dot{f}_{0,E}(i)|})\frac{\dot{f}_{0,E}(i)}{|\dot{f}_{0}(i)|}\bigg)\\
    & =\sum_{(E,i):\pi_{E}(i)=v} (-1)^{i+1}\frac{n(E)}{l(E)}\dot{Y}_{E}(i)^{\perp}
\end{align*}
where given $V\in T_{f_{0}(t)}M$ we denote $V^{\perp}$ the projection of $V$ onto the orthogonal complement of the subspace $\langle\dot{f}_{0}(t)\rangle$. Thus we have the second variation formula

\begin{equation*}
    \Hess_{f_{0}}l_{g}(X,Y)=\sum_{E\in\mathscr{E}}\int_{E}g(A_{E}(Y)(t),X(t))dt+\sum_{v\in\mathscr{V}}g(B_{v}(Y),X(v)).
\end{equation*}

We say that a vector field $J$ along $f_{0}$ is Jacobi if $\Hess_{f_{0}} l_{g}(J,X)=0$ for all vector fields $X$ along $f_{0}$. By the second variation formula, $J$ is Jacobi along $f_{0}$ if and only if
\begin{enumerate}
    \item $J$ verifies the Jacobi equation $\ddot{J}^{\perp}+R(\dot{f}_{0}(t),J(t)^{\perp})\dot{f}_{0}(t)=0$ along each $E\in\mathscr{E}$.
    \item $B_{v}(J)=0$ for all $v\in\mathscr{V}$.
\end{enumerate}

\begin{definition}
We say that a vector field $X\in\mathfrak{X}^{2}(f_{0})$ is parallel if its restriction to each edge $E\in\mathscr{E}$ is a parallel vector field along the corresponding geodesic segment.
\end{definition}

\begin{rk}
By the second variation formula, any parallel vector field along $f_{0}$ is automatically Jacobi.
\end{rk}

\begin{definition}
A stationary geodesic network $f_{0}\in\Omega(\Gamma,M)$ with respect to a metric $g\in\mathcal{M}$ is nondegenerate if every Jacobi field $J$ along $f_{0}$ is parallel.
\end{definition}

\begin{definition}
Given a weighted multigraph $\Gamma$ and a Riemannian metric $g\in\mathcal{M}^{k}$, $g$ is said to be bumpy with respect to $\Gamma$ if every stationary geodesic network $f\in\Omega^{emb}(\Gamma,M)$ with respect to $g$ is nondegenerate. A Riemannian metric $g\in\mathcal{M}^{k}$ is said to be bumpy if it is bumpy with respect to $\Gamma$ for every good weighted multigraph $\Gamma$.
\end{definition}

\section{$C^{0}$ Banach manifold structure for the space of immersed paths under reparametrizations}\label{paths}
Consider the space $\Omega([0,1],M)$ of $C^{2}$ immersions $f:[0,1]\to M$, where $M$ is an $n$-dimensional smooth manifold provided with an auxiliary smooth Riemannian metric $\gamma_{0}$. Denote
\begin{equation*}
    \Diff_{2}([0,1])=\{\tau:[0,1]\to[0,1]:\tau\text{ is a }C^{2}\text{ diffeomorphism},\tau(0)=0,\tau(1)=1\}
\end{equation*}
Define an equivalence relation $\sim$ on $\Omega([0,1],M)$ as $f\sim g$ if and only if there exists $\tau\in\Diff_{2}([0,1])$ such that $f=g\circ\tau$. If that happens we will say that $f$ is a reparametrization of $g$. Let $\hat{\Omega}([0,1],M)=\Omega([0,1],M)/\sim$ be the quotient space by the equivalence relation $\sim$ with the quotient topology. The aim of this section is to give a $C^{0}$ Banach manifold structure for $\hat{\Omega}([0,1],M)$ (i.e. an atlas consisting of charts with values in a fixed Banach space whose transition maps are just continuous). Our constructions would also work if we replaced $C^{2}$ regularity by $C^{k}$ regularity for any $k\geq 1$, but we will focus on the case $k=2$ because that is what we are using in the rest of the paper.

\begin{rk}
We only get a $C^{0}$ Banach manifold structure (and not $C^{j}$ for any $j\geq 1$) due to the fact that, as it is shown below, the transition maps involve taking compositions and inverses of $C^{2}$ functions; and the operators $C^{k}(B,C)\times C^{k}(A,B)\to C^{k}(A,C)$, $(g,f)\mapsto g\circ f$ and $\Diff^{k}(A)\to\Diff^{k}(A)$, $f\mapsto f^{-1}$ (where $A$,$B$,$C$ are smooth manifolds) are continuous but not differentiable in the $C^{k}$ topology for $k<\infty$ (see for example \cite[p. 2]{Inci}).
\end{rk}

Let us fix $f_{0}\in\Omega([0,1],M)$. By density of the $C^{\infty}$ immersions, we can assume without loss of generality that $f_{0}$ is $C^{\infty}$ (any $[f_{0}]$ will be in the domain of a chart of $\hat{\Omega}(\Gamma,M)$ centered at $[\Tilde{f}_{0}]$ for some $\Tilde{f}_{0}$ of class $C^{\infty}$). This will allow us to apply the Tubular Neighborhood Theorem and have a $C^{\infty}$ exponential map. We want to describe a neighborhood of $[f_{0}]$ in $\hat{\Omega}([0,1],M)$. Take $\eta>0$ small so that $f_{0}$ can be extended to a $C^{\infty}$ immersion $f_{0}:(-\eta,1+\eta)\to M$. Denote by $N_{f_{0}}$ the normal bundle along $f_{0}:(-\eta,1+\eta)\to M$ and given $s>0$ let $N^{s}_{f_{0}}=\{v\in N_{f_{0}}:|v|_{\gamma_{0}}<s\}$ and $U_{f_{0}}^{s}=\mathcal{E}(N_{f_{0}}^{s})\subseteq M$ (where $\mathcal{E}:TM\to M$ is the exponential map with respect to the auxiliary metric $\gamma_{0}$). By the Tubular Neighborhood Theorem, there exists $r>0$ such that $\mathcal{E}:N_{f_{0}}^{r}\to U_{f_{0}}^{r}$ is a local diffeomorphism (it is actually a diffeomorphism if $f_{0}$ is an embedding).

\begin{lemma}\label{Lemma theta}
There exist a neighborhood $W_{1}$ of $f_{0}$ in $\Omega([0,1],M)$ and a neighborhood $\overline{W}_{2}$ of $h_{0}(t)=(t,0)$ in $\Omega([0,1],N_{f_{0}})$ such that the map $\overline{\Theta}:\overline{W}_{2}\to W_{1}$ defined as $\overline{\Theta}(h)=\mathcal{E}\circ h$ is a diffeomorphism of Banach manifolds. 
\end{lemma}

\begin{proof}
    In case $\mathcal{E}:N^{r}(f_{0})\to U^{r}_{f_{0}}$ is a diffeomorphism, we can define
    \begin{align*}
    W_{1} & =\{f\in\Omega([0,1],M):\image(f)\subseteq U^{r}_{f_{0}}\}=\Omega([0,1],U^{r}_{f_{0}}),\\
    \overline{W}_{2} & =\{v\in\Omega([0,1],N_{f_{0}}):\image(v)\subseteq N_{f_{0}}^{r}\}=\Omega([0,1],N_{f_{0}}^{r})
\end{align*}
  and a map $\overline{\Theta}':W_{1}\to \overline{W}_{2}$ as $\overline{\Theta}'(f)=\mathcal{E}^{-1}\circ f$. Both $\overline{\Theta}$ and $\overline{\Theta}'$ are smooth maps of Banach manifolds, and inverses of each other so we get the desired result. When the immersion $f_{0}$ is not injective, $\overline{\Theta}'(f)=\mathcal{E}^{-1}\circ f$ is not well defined globally, but we can define it locally over a finite collection of intervals $\{I_{i}\}_{1\leq i\leq K}$ covering $[0,1]$ such that $f_{0}$ is injective along $I_{i}$ and $\mathcal{E}$ is a diffeomorphism when restricted to $N_{f_{0}}^{r}|_{I_{i}}$ for each $i\leq i\leq K$ and some $r>0$. By a gluing argument, we can construct a smooth inverse $\overline{\Theta}'$ for $\overline{\Theta}$.
\end{proof}

\begin{cor}\label{Cor phi}
Let $\phi:G=(-\eta,1+\eta)\times\mathbb{R}^{n-1}\to N_{f_{0}}$ be a trivialization of the normal bundle $N_{f_{0}}$. Let $v_{0}:[0,1]\to G$ be the map $v_{0}(t)=(t,0)$. Then taking $W_{1}\subseteq\Omega([0,1],M)$ from the previous lemma, there exists a neighborhood $W_{2}\subseteq\Omega([0,1],G)$ of $v_{0}$ such that $\Theta:W_{2}\to W_{1}$ given by $\Theta(v)=\mathcal{E}\circ\phi\circ v$ is a diffeomorphism of Banach manifolds.
\end{cor}

Therefore it is enough to model a neighborhood of $[v_{0}]\in\hat{\Omega}([0,1],G)$ as an open subset of a Banach manifold.

\begin{prop}\label{Proposition 1}
Given $v\in\Omega([0,1],G)$ let us denote $a_{v}=\pi(v(0))$ and $b_{v}=\pi(v(1))$, where $\pi:G=(-\eta,1+\eta)\times\mathbb{R}^{n-1}\to(-\eta,1+\eta)$ is the projection onto the first coordinate. There exists an open neighborhood $v_{0}\in W_{3}\subseteq W_{2}\subseteq\Omega([0,1],G)$ with the following property: for every $v\in W_{3}$ there exists a section $\tilde{v}$ of $G|_{[a_{v},b_{v}]}$ such that the map $\tilde{v}:[a_{v},b_{v}]\to G$ is a reparametrization of $v$.
\end{prop}

\begin{proof}
Take $\delta<\eta$ such that $W_{3}:=\{v\in\Omega([0,1],G):\Vert v-v_{0}\Vert_{2}<\delta\}$ is contained in $W_{2}$ and each $v\in W_{3}$ is an embedding. Assume also that $\delta<\frac{1}{7}$. Pick $v\in W_{3}$. First, we want to prove that $\pi(v([0,1]))=[a_{v},b_{v}]$. Notice that it suffices to show $\pi(v([0,1]))\subseteq [a_{v},b_{v}]$. Define $v_{1}:[0,1]\to G$ by  $v_{1}(t)=((1-t)a_{v}+tb_{v},0)$. Consider the map $w=\pi\circ v:[0,1]\to\mathbb{R}$ which is just the first component of $v$. We claim that $w'(t)\geq 0$ for all $t\in[0,1]$. Suppose not. Then there exists $t_{0}\in[0,1]$ such that $w'(t_{0})<0$. Therefore,
\begin{equation*}
    |v'(t_{0})-v_{1}'(t_{0})|\geq |\pi(v'(t_{0})-v_{1}'(t_{0}))|=|w'(t_{0})-(b_{v}-a_{v})|\geq b_{v}-a_{v}>1-2 \delta.
\end{equation*}
On the other hand, it is easy to see that $\Vert v_{0}-v_{1} \Vert_{2}<2(|a_{v}|+|b_{v}-1|)<4\delta$, then from the previous
\begin{equation*}
    \Vert v-v_{0}\Vert_{2}\geq \Vert v-v_{1}\Vert_{2}-\Vert v_{1}-v_{0}\Vert_{2}>(1-2\delta)-4\delta=1-6\delta>\delta
\end{equation*}
as $\delta<\frac{1}{7}$, which is a contradiction because we assumed $\Vert v-v_{0}\Vert_{2}<\delta$. Then, $w'(t)\geq 0$ for all $t\in[0,1]$ and as $w(0)=a_{v}$ we deduce $w(t)\geq a_{v}$ for all $t\in[0,1]$. Analogously, $w(t)\leq b_{v}$ for all $t\in [0,1]$ and hence $\pi(v([0,1]))\subseteq [a_{v},b_{v}]$ as desired.

Given $a,b\in(-\eta,1+\eta)$ with $a<b$ define $\tau_{ab}:[a,b]\to [0,1]$ as $\tau_{ab}(t)=\frac{t-a}{b-a}$. Notice that $\tau_{ab}$ is the inverse of $\chi_{ab}:[0,1]\to[a,b]$ given by $\chi_{ab}(t)=(1-t)a+tb$. By the previous, each $v\in W_{3}$ induces a smooth function $\theta_{v}:=\tau_{a_{v}b_{v}}\circ\pi\circ v:[0,1]\to[0,1]$. Explicitly, $\theta_{v}(t)=\frac{\pi(v(t))-\pi(v(0))}{\pi(v(1))-\pi(v(0))}$. As $\theta_{v_{0}}=id$, shrinking $\delta$ again if necessary we can assume that $v\in W_{3}$ implies $\theta_{v}:[0,1]\to[0,1]$ is a $C^{2}$ diffeomorphism fixing $0$ and $1$. In that case, $\pi\circ v$ and hence $\pi:v([0,1])\to [a_{v},b_{v}]$ are diffeomorphisms. If we denote $\tilde{v}:[a_{v},b_{v}]\to v([0,1])$ the inverse of $\pi:v([0,1])\to[a_{v},b_{v}]$, then $\tilde{v}$ is a section of $G|_{[a_{v},b_{v}]}$ and we have $v=\tilde{v}\circ \chi_{a_{v}b_{v}}\circ\theta_{v}$ being $v$ a reparametrization of $\tilde{v}$.
\end{proof}

The previous tells us that if we take $a\in(-\delta,\delta)$, $b\in(1-\delta,1+\delta)$ and $u\in C^{2}([0,1],\mathbb{R}^{n-1})$ where $\delta$ satisfies the requirements from above, we can define a map $v_{abu}:[0,1]\to G$ as $v_{abu}(t)=((1-t)a+tb,u(t))$ so that every $v\in W_{3}$ is a reparametrization of some $v_{abu}$. Specifically, given $v\in W_{3}$ if $v=\tilde{v}\circ\chi_{a_{v}b_{v}}\circ\theta_{v}$ as above and $\tilde{v}(s)=(s,\tilde{u}(s))$ then we must choose $a=a_{v}$, $b=b_{v}$ and $u=\tilde{u}\circ\chi_{a_{v}b_{v}}$. Consider the map $\Xi:(-\delta,\delta)\times(1-\delta,1+\delta)\times C^{2}([0,1],\mathbb{R}^{n-1})\to\Omega([0,1],G)$ given by $\Xi(a,b,u)=v_{abu}$. Denote $p:\Omega([0,1],G)\to\hat{\Omega}([0,1],G)$ the projection map.

\begin{lemma}
The map $\Xi':W_{3}\to (-\delta,\delta)\times(1-\delta,1+\delta)\times C^{2}([0,1],\mathbb{R}^{n-1})$ given by $\Xi'(v)=(a_{v},b_{v},u_{v})$ with $u_{v}=\tilde{u}\circ\chi_{a_{v}b_{v}}$ as described before is continuous.
\end{lemma}

\begin{proof}
It is enough to show that $\Xi'_{3}:W_{3}\to C^{2}([0,1],\mathbb{R}^{n-1})$ defined as $\Xi'_{3}(v)=u_{v}$ is continuous. Let $\tilde{\pi}:G\to\mathbb{R}^{n-1}$ be the projection onto the last $n-1$ coordinates so that $\tilde{u}=\tilde{\pi}\circ\tilde{v}$. We have
\begin{equation*}
    u_{v}=\tilde{u}\circ\chi_{a_{v}b_{v}}=\tilde{\pi}\circ\tilde{v}\circ\chi_{a_{v}b_{v}}=\tilde{\pi}\circ v\circ\theta_{v}^{-1}\circ\tau_{a_{v}b_{v}}\circ\chi_{a_{v}b_{v}}=\tilde{\pi}\circ v\circ\theta_{v}^{-1}.
\end{equation*}
But $v\mapsto\theta_{v}^{-1}$ is continuous because so is $v\mapsto\theta_{v}$ and $\theta\mapsto\theta^{-1}$ for $\theta\in\Diff_{2}([0,1])$. Therefore by continuity of the composition, $\Xi'$ is continuous. Notice that the differentiability  fails as we are precomposing with and taking inverses of $C^{2}$ maps, which are not differentiable operations on spaces of $C^{2}$ functions (\cite[p. 2]{Inci}).
\end{proof}

\begin{rk}\label{Rk xi}
Given $v\in W_{3}$ we have $p\circ\Xi\circ\Xi'(v)=p(v)$.
\end{rk}

\begin{lemma}
Define $W_{4}=\Xi^{-1}(W_{3})\subseteq\mathbb{R}^{2}\times C^{2}([0,1],\mathbb{R}^{n-1})$. Then $p\circ\Xi:W_{4}\to\hat{\Omega}([0,1],G)$ is injective.
\end{lemma}

\begin{proof}
Suppose $p\circ\Xi(a_{1},b_{1},u_{1})=p\circ\Xi(a_{2},b_{2},u_{2})$ for some $(a_{1},b_{1},u_{1}),(a_{2},b_{2},u_{2})\in U$. Denote $v_{i}=v_{a_{i}b_{i}u_{i}}\in W_{3}$ for $i=1,2$; being $[v_{1}]=[v_{2}]$. Then $v_{1}$ and $v_{2}$ have the same image, so $\pi\circ v_{1}([0,1])=\pi\circ v_{2}([0,1])$ which means $[a_{1},b_{1}]=[a_{2},b_{2}]$ hence $a_{1}=b_{1}$ and $a_{2}=b_{2}$. Therefore, $v_{1}(t)=(a_{1}(1-t)+b_{1}t,u_{1}(t))$ is a reparametrization of $v_{2}(t)=(a_{1}(1-t)+b_{1}t,u_{2}(t))$. By looking at the first coordinate we deduce that the reparametrization must be just composing with the identity, and hence $u_{1}=u_{2}$.
\end{proof}

\begin{lemma}\label{Lemma section}
There exists an open neighborhood $W_{5}\subseteq W_{4}\subseteq(-\delta,\delta)\times(1-\delta,1+\delta)\times C^{2}([0,1],\mathbb{R}^{n-1})$ of $(0,1,0)$ such that $\Xi'\circ\Xi(a,b,u)=(a,b,u)$ for all $(a,b,u)\in W_{5}$.
\end{lemma}

\begin{proof}
First of all observe that $\Xi(0,1,0)=v_{0}:t\mapsto(t,0)$ and by definition $\Xi'(v_{0})=(0,1,0)$ so $\Xi'\circ\Xi(0,1,0)=(0,1,0)$. Set $W_{5}=W_{4}\cap(\Xi'\circ\Xi)^{-1}(W_{4})$, by continuity of $\Xi$ and $\Xi'$ and the previous observation $W_{5}$ is an open neighborhood of $(0,1,0)$. Given $(a,b,u)\in W_{5}$ let $(\tilde{a},\tilde{b},\tilde{u})=\Xi'\circ\Xi(a,b,u)\in W_{4}$. Then $p\circ\Xi(\tilde{a},\tilde{b},\tilde{u})=p\circ\Xi\circ\Xi'\circ\Xi(a,b,u)=p\circ\Xi(a,b,u)$ because of Remark \ref{Rk xi} and the fact that $\Xi(a,b,u)\in W_{3}$. As $(a,b,u),(\tilde{a},\tilde{b},\tilde{u})\in W_{4}$ and $p\circ\Xi|_{W_{4}}$ is injective we deduce $(a,b,u)=(\tilde{a},\tilde{b},\tilde{u})$.
\end{proof}

Let us provide $\mathbb{R}\times\mathbb{R}\times C^{2}([0,1],\mathbb{R}^{n-1})$ with the norm $\Vert (a,b,u)\Vert=|a|+|b|+\Vert u\Vert_{2}$ making it a Banach space. Notice that $\Xi:\mathbb{R}\times\mathbb{R}\times C^{2}([0,1],\mathbb{R}^{n-1})\to C^{2}([0,1],\mathbb{R}^{n})$ is linear and
\begin{equation*}
    \frac{1}{3}\Vert (a,b,u)\Vert\leq\Vert v_{abu}\Vert_{2}=\Vert\Xi(a,b,u)\Vert_{2}\leq 2\Vert(a,b,u)\Vert
\end{equation*}
therefore $C^{2}_{*}([0,1],\mathbb{R}^{n}):=\image(\Xi)\subseteq C^{2}([0,1],\mathbb{R}^{n})$ is a closed subspace of $ C^{2}([0,1],\mathbb{R}^{n})$ and by the Open Mapping Theorem $\Xi:\mathbb{R}\times\mathbb{R}\times C^{2}([0,1],\mathbb{R}^{n-1})\to C^{2}_{*}([0,1],\mathbb{R}^{n})$ is an isomorphism of Banach spaces.

\begin{thm}\label{homeo thm}
The subset $\hat{W_{5}}:=p\circ\Xi(W_{5})\subseteq\hat{\Omega}([0,1],G)$ is open and $p\circ\Xi:W_{5}\to\hat{W}_{5}$ is a homeomorphism.
\end{thm}

\begin{proof}
Let us start by showing that $p\circ\Xi:W_{5}\to\hat{\Omega}([0,1],G)$ is an open map. Let $V\subseteq W_{5}$ be an open subset. Then $V':=\Xi(V)\subseteq C^{2}_{*}([0,1],\mathbb{R}^{n})\cap W_{3}$ is an open subset of $C^{2}_{*}([0,1],\mathbb{R}^{n})$. Define $W'=(\Xi\circ\Xi')^{-1}(V')\cap W_{3}\subseteq (\Xi\circ\Xi')^{-1}(W_{3})\cap W_{3}$ which is an open subset of $\Omega([0,1],G)\subseteq C^{2}([0,1],G)$. If $v\in V'$ then $v=\Xi(a,b,u)$ for some $(a,b,u)\in W_{5}$ therefore $\Xi\circ\Xi'(v)=\Xi\circ\Xi'\circ\Xi(a,b,u)=\Xi(a,b,u)=v$ by Lemma \ref{Lemma section} and hence $v\in(\Xi\circ\Xi')^{-1}(V')\cap W_{3}=W'$. This means that $V'\subseteq W'$ and hence $p(V')\subseteq p(W')$. But conversely, given $v\in W'$ by definition $v\in W_{3}$ so $p(v)=p\circ\Xi\circ\Xi'(v)\in p(V')$ because $\Xi\circ\Xi'(v)\in V'$. Therefore $p(V')=p(W')$ and as $p$ is an open map we deduce that $p(V')=p\circ\Xi(V)$ is open, as desired.

Therefore, as $p\circ\Xi:W_{5}\to\hat{\Omega}([0,1],G)$ is continuous, open and injective, it is a homeomorphism onto its image $\hat{W}_{5}$ as we wanted.
\end{proof}

We can use the results of this section to construct an atlas for $\hat{\Omega}([0,1],M)$ with charts of the form $(\hat{W}_{5},(p\circ\Theta\circ\Xi)^{-1})$ centered at $C^{\infty}$ immersions $[f_{0}]$ (with $\hat{W}_{5}$ as in Theorem \ref{homeo thm}), which yields a $C^{0}$ Banach manifold structure modeled by $\mathbb{R}\times\mathbb{R}\times C^{2}([0,1],\mathbb{R}^{n-1})$.

\section{The length functional on the space $\hat{\Omega}(\Gamma,M)$}\label{lengthsgn}

Let us fix a good* weighted multigraph $\Gamma$ (i.e. $\Gamma$ is connected and every vertex $v$ of $\Gamma$ has at least three different incoming edges). We define an equivalence relation $\sim$ in $\Omega(\Gamma,M)$ as follows: $f_{0}\sim f_{1}$ if and only if there exists a homeomorphism $\theta:\Gamma\to\Gamma$ such that
\begin{enumerate}
    \item $\theta(v)=v$ for all $v\in\mathscr{V}$.
    \item $\theta(E)=E$ for all $E\in\mathscr{E}$ and moreover $\theta_{E}:=\theta|_{E}:E\to E$ is a $C^{2}$ diffeomorphism.
    \item $f_{1}=f_{0}\circ\theta$.
\end{enumerate}
We consider the quotient space $\hat{\Omega}(\Gamma,M)=\Omega(\Gamma,M)/\sim$ with the quotient topology. Define the space $\Omega(\mathscr{E},M)=\prod_{E\in\mathscr{E}}\Omega(E,M)$ ($\Omega(E,M)\cong\Omega([0,1],M)$ by identifying $E\cong[0,1]$) being the map $\iota:\Omega(\Gamma,M)\to\Omega(\mathscr{E},M)$ defined as $\iota(f)=(f_{E})_{E\in\mathscr{E}}$ a subspace map. We can also consider an equivalence relation $\sim$ in $\Omega(\mathscr{E},M)$ as follows: $f=(f_{E})_{E\in\mathscr{E}}\sim g=(g_{E})_{E\in\mathscr{E}}$ if there exists $\theta=(\theta_{E})_{E\in\mathscr{E}}\in\prod_{E\in\mathscr{E}}\Diff_{2}(E)$ such that $\theta_{E}$ fixes the vertices of $E$ and $g_{E}=f_{E}\circ\theta_{E}$ for all $E\in\mathscr{E}$. As before, we define the quotient space $\hat{\Omega}(\mathscr{E},M)=\Omega(\mathscr{E},M)/\sim\cong\prod_{E\in\mathscr{E}}\hat{\Omega}(E,M)$ with the quotient topology. It is clear that $\iota$ descends to a subspace map $\hat{\iota}:\hat{\Omega}(\Gamma,M)\to\hat{\Omega}(\mathscr{E},M)$.

Observe that the space $\hat{\Omega}(\mathscr{E},M)$ has a product $C^{0}$ manifold structure modeled on the Banach space $\mathcal{B}=\prod_{E\in\mathscr{E}}\mathbb{R}^{2}\times C^{2}(E,\mathbb{R}^{n-1})$. We proceed to describe the atlas induced by this product structure. Let $f=(f_{E})_{E\in\mathscr{E}}\in\Omega(\mathscr{E},M)$ be such that $f_{E}$ is $C^{\infty}$ for every $E\in\mathscr{E}$. For each $f_{E}$ we do the constructions of the previous section, i.e. we consider:
\begin{enumerate}
    \item A trivialization $\phi_{E}:(-\eta_{E},1+\eta_{E})\times\mathbb{R}^{n-1}\to N_{f_{E}}$ of $N_{f_{E}}|_{(-\eta_{E},1+\eta_{E})}$.
    \item Open sets $f_{E}\in W_{1}(f_{E})\subseteq\Omega(E,M)$, $W_{3}(f_{E})\subseteq W_{2}(f_{E})\subseteq\Omega([0,1],G)$ and $W_{5}(f_{E})\subseteq W_{4}(f_{E})\subseteq (-\eta_{E},\eta_{E})\times(1-\eta_{E},1+\eta_{E})\times C^{2}([0,1],\mathbb{R}^{n-1})$ with the properties described in the previous section. In particular, $p\circ\Theta_{E}\circ\Xi:W_{5}(f_{E})\to \hat{W}_{5}(f_{E})=p\circ\Theta_{E}\circ\Xi(W_{5}(f_{E}))$ is a homeomorphism.
    \item A real number $\delta_{E}>0$ such that $U_{E}:=(-\delta_{E},\delta_{E})\times(1-\delta_{E},1+\delta_{E})\times C^{2}([0,1],\mathbb{R}^{n-1})_{\delta_{E}}\subseteq W_{5}(f_{E})$.
\end{enumerate}
In the previous, we used the notation
\begin{equation*}
    C^{2}([0,1],\mathbb{R}^{n-1})_{\alpha}:=\{u\in C^{2}([0,1],\mathbb{R}^{n-1}):\Vert u\Vert_{2}<\alpha\}.
\end{equation*}
Denote $\hat{U}_{E}=p\circ\Theta_{E}\circ\Xi(U_{E})\subseteq\hat{\Omega}(E,M)$, $U=\prod_{E\in\mathscr{E}}U_{E}\subseteq\mathcal{B}$ and $\hat{U}=\prod_{E\in\mathscr{E}}\hat{U}_{E}\subseteq\hat{\Omega}(\mathscr{E},M)$. From the previous section, we have homeomorphisms $\Lambda_{E}:U_{E}\to\hat{U}_{E}$ defined as $\Lambda_{E}(u)=p(\Theta_{E}(\Xi(u)))$ and they induce a homeomorphism $\Lambda=\prod_{E\in\mathscr{E}}\Lambda_{E}:U\to\hat{U}$. We define $\tilde{\Lambda}_{E}:U_{E}\to\Omega(E,M)$ as $\tilde{\Lambda}_{E}(u)=\Theta_{E}\circ\Xi(u)$ and $\tilde{\Lambda}:U\to\prod_{E\in\mathscr{E}}\Omega(E,M)$ as $\tilde{\Lambda}(u)=(\tilde{\Lambda}(u_{E}))_{E\in\mathscr{E}}$. Denote $\Sigma=\Lambda^{-1}$. Then $(\hat{U},\Sigma)$ is a chart of $\hat{\Omega}(\mathscr{E},M)$ at $f$ for the product structure we are considering, and the collection of all such charts is a $C^{0}$ atlas of $\hat{\Omega}(\mathscr{E},M)$.

\begin{prop}\label{Submanifold prop}
$\hat{\Omega}(\Gamma,M)\subseteq\hat{\Omega}(\mathscr{E},M)$ is an embedded $C^{0}$ Banach submanifold, and its image under any chart $(\hat{U},\Sigma)$ as constructed above is a smooth Banach submanifold of $\mathcal{B}$.
\end{prop}

We introduce the following notation which will be useful in the proof of the proposition and in the rest of the section.

\begin{notation}
Given an edge $E\in\mathscr{E}$ we denote $c_{0}(E)=a_{E}$ and $c_{1}(E)=b_{E}$.
\end{notation}

\begin{definition}
Given a vertex $v\in\mathscr{V}$, we denote $m(v)$ the number of incoming edges of the graph $\Gamma$ at $v$, i.e. the number of pairs $(E,i)\in\mathscr{E}\times\{0,1\}$ such that $\pi_{E}(i)=v$ (as in Definition \ref{incoming edge}, notice that loops at $v$ count twice as incoming edges). For each vertex $v$, we choose a preferred pair $(E_{v},i_{v})\in\mathscr{E}\times\{0,1\}$ such that $\pi_{E_{v}}(i_{v})=v$.
\end{definition}

\begin{rk}
    Notice that $\sum_{v\in\mathscr{V}}m(v)=2|\mathscr{E}|$.
\end{rk}

\begin{proof}[Proof of Proposition \ref{Submanifold prop}]
Let $f_{0}\in \Omega(\Gamma,M)$ be a $\Gamma$-net which is $C^{\infty}$ along the edges. Consider a chart $(\hat{U},\Sigma)$ at $[f_{0}]$ of the product manifold $\hat{\Omega}(\mathscr{E},M)=\prod_{E\in\mathscr{E}}\hat{\Omega}(E,M)$ as constructed before. We are going to describe $\Sigma(\hat{\Omega}(\Gamma,M)\cap\hat{U})$ as an embedded Banach submanifold of $U$.

In order to do that, we need to understand which elements $u=(a_{E},b_{E},u_{E})_{E\in\mathscr{E}}$ verify $\tilde{\Lambda}(u)\in\Omega(\Gamma,M)$. Notice that $(c_{i}(E),u_{E}(i))$ is equal to $(a_{E},u_{E}(0))$ if $i=0$ and to $(b_{E},u_{E}(1))$ if $i=1$. Observe that $u\in U$ represents a map which is continuous at $v$ if and only if given $(E,i)\in\mathscr{E}\times\{0,1\}$ such that $\pi_{E}(i)=v$ we have
\begin{equation*}
    \mathcal{E}\circ\phi_{E}(c_{i}(E),u_{E}(i))=\mathcal{E}\circ\phi_{E_{v}}(c_{i_{v}}(E_{v}),u_{E_{v}}(i_{v})).
\end{equation*}
We know that the map $\mathcal{E}\circ\phi_{E_{v}}$ is a diffeomorphism in a small neighborhood of $(i_{v},0)$. Let us denote its inverse as $(\mathcal{E}\circ\phi_{E_{v}})^{-1}$. Define $C_{v}:U\to(\mathbb{R}^{n})^{m(v)-1}$ as
\begin{equation*}
    C_{v}(u)=\big((\mathcal{E}\circ\phi_{E_{v}})^{-1}\circ \mathcal{E}\circ\phi_{E}(c_{i}(E),u_{E}(i))-(c_{i_{v}}(E_{v}),u_{E_{v}}(i_{v}))\big)_{(E,i)\neq(E_{v},i_{v}):\pi_{E}(i)=v}.
\end{equation*}
Then $u$ represents a map which is continuous at $v$ if and only if $C_{v}(u)=0$. Denote $C:U\to\prod_{v\in\mathscr{V}}(\mathbb{R}^{n})^{m(v)-1}=(\mathbb{R}^{n})^{2|\mathscr{E}|-|\mathscr{V}|}$ the smooth map defined as $C(u)=(C_{v}(u))_{v\in\mathscr{V}}$. From the previous we see that $\Lambda(u)\in\hat{\Omega}(\Gamma,M)$ if and only if $C(u)=0$.

\begin{lemma}\label{Submersion lemma}
$C^{-1}(0)$ is an embedded smooth Banach submanifold of $U$.
\end{lemma}

\begin{proof}[Proof of the lemma]
Let $u\in U$ be such that $C(u)=0$. Let $C_{v}^{(E,i)}$ be the component of $C_{v}$ corresponding to $(E,i)$. Denote $\{e_{j}\}_{1\leq j\leq n}$ the canonical basis of $\mathbb{R}^{n}$. Consider the basis $\{e_{j}^{v,(E,i)}:1\leq j\leq n,v\in\mathscr{V},(E,i)\neq(E_{v},i_{v}):\pi_{E}(i)=v\}$ of $\prod_{v\in\mathscr{V}}(\mathbb{R}^{n})^{m(v)-1}$. Given $v\in\mathscr{V}$, $(E,i)\in\mathscr{E}\times\{0,1\}$ such that $\pi_{E}(i)=v$ and $(E,i)\neq(E_{v},i_{v})$, and $1\leq j\leq n$ we will construct a one parameter family $\{u_{s}\}_{s\in(-1,1)}$ in $U$ such that $u_{0}=u$, $\frac{d}{ds}|_{s=0}C_{v}^{(E,i)}(u_{s})=e_{j}$ and $\frac{d}{ds}|_{s=0}C_{v'}^{(E',i')}(u_{s})=0$ for all $(v',(E',i'))\neq(v,(E,i))$. Therefore if $X(v,(E,i),j)=\frac{d}{ds}|_{s=0}u_{s}$ by definition we will have $DC_{u}(X(v,(E,i),j))=e_{j}^{v,(E,i)}$.

The construction is as follows. Let $\rho$ be a bump function on $E$ which is zero outside a small interval $I$ around $i$ and which takes the value $1$ at $i$. Let $a\in\mathbb{R}$ and $w\in\mathbb{R}^{n-1}$ be such that $D((\mathcal{E}\circ\phi_{E_{v}})^{-1}\circ \mathcal{E}\circ\phi_{E})_{(c_{i}(E),u_{E}(i))}(a,w)=e_{j}$. Define
\begin{align*}
u_{s,E}(t)=u_{E}(t)+s\rho(t)w
\end{align*}
and $u_{s,E'}(t)=u_{E'}(t)$ for all $E'\neq E$. Define $c_{s,i}(E)=c_{i}(E)+sa$ and $c_{s,i'}(E')=c_{i'}(E')$ for all $(E',i')\neq (E,i)$. Then $u_{s}=(a_{s,E},b_{s,E},u_{s,E})_{E\in\mathscr{E}}$ is a smooth one parameter family with $u_{0}=u$ and
\begin{equation*}
    \frac{d}{ds}\bigg|_{s=0}C_{v}^{(E,i)}(u_{s}) =D((\mathcal{E}\circ\phi_{E_{v}})^{-1}\circ \mathcal{E}\circ\phi_{E})_{(c_{i}(E),u_{E}(i))}(a,w)=e_{j}.
\end{equation*}
Also as $C_{v'}^{(E',i')}(u_{s})=C_{v'}^{(E',i')}(u)$ for all $(v',(E',i'))\neq (v,(E,i))$ we deduce \linebreak $\frac{d}{ds}|_{s=0}C_{v'}^{(E',i')}(u_{s})=0$ in that case.

Therefore we have a collection of vectors $\{X(v,(E,i),j)\}$ such that \linebreak $DC_{u}(X(v,(E,i),j))=e_{j}^{v,(E,i)}$. Observe that their images under $DC_{u}$ form a basis of $\prod_{v\in\mathscr{V}}(\mathbb{R}^{n})^{m(v)-1}$ and hence $DC_{u}$ is surjective. Denote $S\subseteq\mathcal{B}= \prod_{E\in\mathscr{E}}\mathbb{R}^{2}\times C^{2}(E,\mathbb{R}^{n-1})$ the span of $\{X(v,(E,i),j)\}$. Then $S$ is finite dimensional and hence closed, and by linear algebra $\ker (DC_{u})\oplus S=\mathcal{B}$. Thus $0$ is a regular value of $C$ and we can apply the Implicit Function Theorem to deduce that $\Sigma(\hat{\Omega}(\Gamma,M)\cap\hat{U})=C^{-1}(0)$ is a smooth Banach submanifold of $U$.
\end{proof}

So far we have shown that $\hat{\Omega}(\Gamma,M)\cap\hat{U}$ is a $C^{0}$ embedded Banach submanifold of $\hat{U}$ modelled in the space $\ker(DC_{u_{0}})$ where $\Lambda(u_{0})=[f_{0}]\in\hat{\Omega}(\Gamma,M)\cap\hat{U}$. As $\ker(DC_{u_{0}})$ has codimension $n(\sum_{v\in\mathscr{V}}m(v)-1)=n(2|\mathscr{E}|-|\mathscr{V}|)$ for all possible $u_{0}$, we deduce that the space modelling $\hat{\Omega}(\Gamma,M)$ locally is independent of the chart $(\hat{U},\Sigma)$ (this is because two closed subspaces of a Banach space with the same finite codimension are isomorphic). Therefore, $\hat{\Omega}(\Gamma,M)\subseteq\hat{\Omega}(\mathscr{E},M)$ is a $C^{0}$ Banach submanifold whose image under any chart $(\hat{U},\Sigma)$ as constructed above is a smooth Banach submanifold of $\mathcal{B}$. 
\end{proof}

\begin{definition}
Following the constructions in the previous proof, we will denote $C_{0}=C^{-1}(0)=\Sigma(\hat{\Omega}(\Gamma,M)\cap\hat{U})$, being $C_{0}\subseteq U$ a Banach submanifold.
\end{definition}

\begin{rk}
All the Banach manifolds previously defined are second countable. This is because they can be obtained from $C^{2}([0,1],M)$ and $\mathbb{R}$ by taking products, quotients and topological subspaces. The same holds for the Banach manifold $\mathcal{M}^{k}$ of $C^{k}$ Riemannian metrics on $M$. These facts are consequences of the following: given a compact manifold $M_{1}$, a smooth manifold $M_{2}$ and a natural number $k\geq 1$; the space $C^{k}(M_{1},M_{2})$ with the $C^{k}$ compact-open topology is metrizable and has a countable base, as it is explained in \cite[p.~35]{Hirsch}.
\end{rk}

We will use the $C^{0}$ Banach submanifold structure of $\hat{\Omega}(\Gamma,M)$ in $\hat{\Omega}(\mathscr{E},M)$ and the particular adapted charts under the atlas $\{(\hat{U},\Lambda)\}$ described above to derive the first and second variation formulas for the length functional in local coordinates, following \cite{White}. The formulas that we will obtain will be analogous to those presented in Section \ref{setup}, the advantage of this approach is that it allows us to use techniques from Differential Equations and Functional Analysis to give a geometric structure to the space of stationary geodesic networks for varying Riemannian metrics. Fix $f_{1}\in\Omega(\Gamma,M)$ and a chart $(\hat{U},\Sigma)$ of $\hat{\Omega}(\mathscr{E},M)$ centered at $[f_{1}]$ as constructed above. Let $g$ be a Riemannian metric on $M$. Given $u\in U$ we define $l_{g}(u)=l_{g}(\Lambda(u))=l_{g}(f)$. Then by definition of $g$-length,
\begin{equation*}
l_{g}(u)=\int_{\Gamma}\sqrt{g_{f(t)}(\dot{f}(t),\dot{f}(t))}dt=\sum_{E\in\mathscr{E}}n(E)\int_{E}\sqrt{g_{f(t)}(\dot{f}(t),\dot{f}(t))}dt
\end{equation*}
and by definition of $\Lambda$,
\begin{align*}
    f_{E}(t) & =\mathcal{E}\circ\phi_{E}(a_{E}(1-t)+b_{E}t,u_{E}(t)),\\
    \dot{f}_{E}(t) & = d(\mathcal{E}\circ\phi_{E})_{v_{E}(t)}(b_{E}-a_{E},\dot{u}_{E}(t))
\end{align*}
where $v_{E}(t)=(a_{E}(1-t)+b_{E}t,u_{E}(t))$. Therefore if we define $F^{E}_{g}:[(-\delta_{E},1+\delta_{E})\times\mathbb{R}^{n-1}]\times\mathbb{R}^{n}\to\mathbb{R}$ as
\begin{equation*}
    F_{g}^{E}(v,w)=\sqrt{g_{\mathcal{E}\circ\phi_{E}(v)}(d(\mathcal{E}\circ\phi_{E})_{v}w,d(\mathcal{E}\circ\phi_{E})_{v}w)}
\end{equation*}
and $\rho:E\times(-\delta_{E},\delta_{E})\times(1-\delta_{E},1+\delta_{E})\times\mathbb{R}^{n-1}\times\mathbb{R}^{n-1}\to [(-\delta_{E},1+\delta_{E})\times\mathbb{R}^{n-1}]\times\mathbb{R}^{n}$ as
\begin{equation*}
    \rho(t,a,b,u,w)=((a(1-t)+bt),u),(b-a,w))
\end{equation*}
then if $L_{g}^{E}=F_{g}^{E}\circ\rho$ it is clear that
\begin{equation*}
    l_{g}(u)=\sum_{E\in\mathscr{E}}n(E)\int_{E}L^{E}_{g}(t,a_{E},b_{E},u_{E}(t),\dot{u}_{E}(t))dt.
\end{equation*}
Now if we have a one parameter family $f:(-\varepsilon,\varepsilon)\to\hat{\Omega}(\Gamma,M)\cap\hat{U}$ and consider $u_{s}=\Lambda^{-1}(f(s))=(a_{E}(s),b_{E}(s),u_{s,E})_{E\in\mathscr{E}}$, then
\begin{align*}
    \frac{d}{ds}\bigg|_{s=0}l_{g}(u_{s}) = & \sum_{E\in\mathscr{E}}n(E)\int_{E}\frac{d}{ds}\bigg|_{s=0}L^{E}_{g}(t,a_{E}(s),b_{E}(s),u_{s,E}(t),\dot{u}_{s,E}(t))dt\\
    = & \sum_{E\in\mathscr{E}}n(E)\int_{E}\frac{\partial L_{g}^{E}}{\partial a}(t,a_{E}(0),b_{E}(0),u_{0,E}(t),\dot{u}_{0,E}(t))a_{E}'(0)dt\\
    & +\sum_{E\in\mathscr{E}}n(E)\int_{E}\frac{\partial L_{g}^{E}}{\partial b}(t,a_{E}(0),b_{E}(0),u_{0,E}(t),\dot{u}_{0,E}(t))b_{E}'(0)dt\\
    & +\sum_{E\in\mathscr{E}}n(E)\int_{E}\nabla_{u}L_{g}^{E}(t,a_{E}(0),b_{E}(0),u_{0,E}(t),\dot{u}_{0,E}(t))\cdot\frac{\partial u_{E}}{\partial s}(0,t)dt\\
    & +\sum_{E\in\mathscr{E}}n(E)\int_{E}\nabla_{w}L_{g}^{E}(t,a_{E}(0),b_{E}(0),u_{0,E}(t),\dot{u}_{0,E}(t))\cdot\frac{\partial^{2} u_{E}}{\partial s\partial t}(0,t)dt
\end{align*}
where $(t,a,b,u,w)\in E\times(-\delta_{E},\delta_{E})\times(1-\delta_{E},1+\delta_{E})\times\mathbb{R}^{n-1}\times\mathbb{R}^{n-1}$ are the $5$ variables on which the function $L_{g}^{E}$ depends. Omitting those variables in the notation and integrating by parts we obtain
\begin{align}\label{FVF 1}
\begin{split}
    \frac{d}{ds}\bigg|_{s=0}l_{g}(u_{s})& =\sum_{E\in\mathscr{E}}n(E)\int_{E}\frac{\partial L_{g}^{E}}{\partial a}a_{E}'(0)+\frac{\partial L_{g}^{E}}{\partial b}b_{E}'(0)+(\nabla_{u}L_{g}^{E}-\frac{d}{dt}\nabla_{w}L_{g}^{E})\cdot\frac{\partial u_{E}}{\partial s}(0,t)dt\\
    & +\sum_{E\in\mathscr{E}}n(E)\nabla_{w}L^{E}_{g}(i,a_{E}(0),b_{E}(0),u_{0,E}(i),\dot{u}_{0,E}(i))\cdot\frac{\partial u_{E}}{\partial s}(0,i)\bigg|_{0}^{1}.
\end{split}
\end{align}
Denote $X(t)=(a_{E}'(0),b_{E}'(0),\frac{\partial u_{E}}{\partial s}(0,t))_{E\in\mathscr{E}}\in T_{u_{0}}C_{0}=\ker (DC_{u_{0}})\subseteq\mathcal{B}$. Define $H^{1,E}_{g}:U\to C^{0}(E,\mathbb{R}^{n-1})$ and $A^{0,E}_{g},A^{1,E}_{g}:U\to\mathbb{R}^{n}$ as

\begin{align*}
H^{1,E}_{g}(u)(t) & =n(E)\bigg( \nabla_{u}L^{E}_{g}(t,a_{E},b_{E},u_{E}(t),\dot{u}_{E}(t))-\frac{d}{dt}\bigg[\nabla_{w}L_{g}^{E}(t,a_{E},b_{E},u_{E}(t),\dot{u}_{E}(t))\bigg]\bigg),\\
A^{0,E}_{g}(u) & =n(E)\bigg(\int_{E}\frac{\partial L_{g}^{E}}{\partial a}(t,a_{E},b_{E},u_{E}(t),\dot{u}_{E}(t))dt,-\nabla_{w}L_{g}^{E}(0,a_{E},b_{E},u_{E}(0),\dot{u}_{E}(0))\bigg),\\
A^{1,E}_{g}(u) & = n(E)\bigg(\int_{E}\frac{\partial L_{g}^{E}}{\partial b}(t,a_{E},b_{E},u_{E}(t),\dot{u}_{E}(t))dt,\nabla_{w}L_{g}^{E}(1,a_{E},b_{E},u_{E}(1),\dot{u}_{E}(1))\bigg).
\end{align*}
According to (\ref{FVF 1}), an element $u_{0}\in C_{0}$ represents a stationary geodesic network if and only if for every $X=(c_{0}(E),c_{1}(E),u_{E})_{E\in\mathscr{E}}\in\ker (DC_{u_{0}})$ it holds
\begin{equation}\label{FVF 2}
    \sum_{E\in\mathscr{E}}\int_{E} H^{1,E}_{g}(u_{0})(t)\cdot u_{E}(t)dt+   \sum_{E\in\mathscr{E}}\sum_{i=0}^{1}A^{i,E}_{g}(u_{0})\cdot (c_{i}(E),u_{E}(i))=0.
\end{equation}
Now observe that the condition $DC_{u_{0}}(X)=0$ implies that given $(E,i)\in\mathscr{E}\times\{0,1\}$ with $\pi_{E}(i)=v$, there exists a linear transformation $T_{v}^{(E,i)}(u_{0}):\mathbb{R}^{n}\to\mathbb{R}^{n}$ such that $(c_{i}(E),u_{E}(i))=T_{v}^{(E,i)}(u_{0})(c_{i_{v}}(E_{v}),u_{E_{v}}(i_{v}))$. Moreover, $DC_{u_{0}}(X)=0$ if and only if $(c_{i}(E),u_{E}(i))=T_{v}^{(E,i)}(u_{0})(c_{i_{v}}(E_{v}),u_{E_{v}}(i_{v}))$ for every $(E,i)\in\mathscr{E}\times\{0,1\}$. Thus we can rewrite
\begin{align*}
  & \sum_{E\in\mathscr{E}}\sum_{i=0}^{1}A^{i,E}_{g}(u_{0})\cdot (c_{i}(E),u_{E}(i))\\
  & =\sum_{v\in\mathscr{V}}\sum_{(E,i):\pi_{E}(i)=v}A^{i,E}_{g}(u_{0})\cdot T_{v}^{(E,i)}(u_{0})(c_{i_{v}}(E_{v}),u_{E_{v}}(i_{v}))\\
  & = \sum_{v\in\mathscr{V}}\sum_{(E,i):\pi_{E}(i)=v}T_{v}^{(E,i)}(u_{0})^{*}(A^{i,E}_{g}(u_{0}))\cdot (c_{i_{v}}(E_{v}),u_{E_{v}}(i_{v}))
\end{align*}
where $T_{v}^{(E,i)}(u_{0})^{*}$ denotes the adjoint of the linear operator $T_{v}^{(E,i)}(u_{0}):\mathbb{R}^{n}\to\mathbb{R}^{n}$ with respect of the Euclidean inner product on $\mathbb{R}^{n}$. Define $H^{2,v}_{g}:U\to\mathbb{R}^{n}$ as
\begin{equation*}
    H^{2,v}_{g}(u)=\sum_{(E,i):\pi_{E}(i)=v}T_{v}^{(E,i)}(u)^{*}(A^{i,E}_{g}(u))
\end{equation*}
and $H^{2}_{g}:U\to(\mathbb{R}^{n})^{|\mathscr{V}|}$ as $H^{2}_{g}(u)=(H^{2,v}_{g}(u))_{v\in\mathscr{V}}$. Then (\ref{FVF 2}) can be rewritten as
\begin{equation}\label{FVF 3}
    \sum_{E\in\mathscr{E}}\int_{E} H^{1,E}_{g}(u_{0})(t)\cdot u_{E}(t)dt+\sum_{v\in\mathscr{V}}H^{2,v}_{g}(u_{0})\cdot(c_{i_{v}}(E_{v}),u_{E_{v}}(i_{v}))=0.
\end{equation}
Define $H^{1}_{g}:U\to\prod_{E\in\mathscr{E}}C^{0}(E,\mathbb{R}^{n-1})$ as $H^{1}_{g}(u)=(H^{1,E}_{g}(u_{E}))_{E\in\mathscr{E}}$.

\begin{prop}
Let $u\in U$. Then $\Lambda(u)$ is a stationary geodesic network with respect to $g\in\mathcal{M}^{k}$ if and only if $H^{1}_{g}(u)=H^{2}_{g}(u)=C(u)=0$.
\end{prop}

\begin{proof}
From (\ref{FVF 3}), it is clear that if $u_{0}\in U$ verifies $C(u_{0})=H^{1}_{g}(u_{0})=H^{2}_{g}(u_{0})=0$ then $\Lambda(u_{0})$ is a stationary geodesic network. We want to see that the converse is also true. Assume $\Lambda(u_{0})$ is stationary. Then $\Lambda(u_{0})\in\hat{\Omega}(\Gamma,M)$ and hence $C(u_{0})=0$. We also know that (\ref{FVF 3}) should hold for every $X=(c_{0}(E),c_{1}(E),u_{E})_{E\in\mathscr{E}}\in\ker (DC_{u_{0}})$.

Suppose $H^{1,E}_{g}(u_{0})$ is not identically zero for some $E\in\mathscr{E}$. Let $t_{0}\in \interior(E)$ be such that $H^{1,E}_{g}(u_{0})(t_{0})=w\neq 0$. Let $u_{E}:E\to\mathbb{R}^{n-1}$ be given by $u_{E}(t)=\rho_{E}(t)w$, where $\rho_{E}:E\to\mathbb{R}_{\geq 0}$ is a $C^{2}$ function such that $\rho_{E}(t_{0})=1$ and $\rho_{E}$ is identically zero outside a small interval $I$ around $t_{0}$ where $H^{1,E}_{g}(u_{0})(t)\cdot w>0$. Let $u_{E'}$ be the identically zero function for all $E'\neq E$. Define $X=(0,0,u_{E'})_{E'\in\mathscr{E}}$. Then as $u_{E'}(0)=u_{E'}(1)=a_{E'}=b_{E'}=0$ for all $E'\in\mathscr{E}$, $X\in\ker(DC_{u_{0}})$. If we plug in $X$ in (\ref{FVF 3}), the second term vanishes and the first term is equal to
\begin{equation*}
    \int_{E}H^{1,E}_{g}(u_{0})(t)\cdot u_{E}(t)dt>0
\end{equation*}
which is a contradiction. Therefore, $H^{1}_{g}(u_{0})$ must be identically zero. Thus we know that for all $X\in\ker(DC_{u_{0}})$
\begin{equation*}
    \sum_{v\in\mathscr{V}}H^{2,v}_{g}(u_{0})\cdot(c_{i_{v}}(E_{v}),u_{E_{v}}(i_{v}))=0.
\end{equation*}
As given any vector $(c_{v},u_{v})_{v\in\mathscr{V}}\in(\mathbb{R}^{n})^{|\mathscr{V}|}$ there exists $X\in\ker (DC_{u_{0}})$ such that $(c_{i_{v}}(E_{v}),u_{E_{v}}(i_{v}))=(c_{v},u_{v})$ for all $v\in\mathscr{V}$ we deduce that $H^{2}_{g}(u_{0})=0$.
\end{proof}

Let us define $H:\mathcal{M}\times C_{0}\to \mathcal{Y}=\Big[\prod_{E\in\mathscr{E}}C^{0}(E,\mathbb{R}^{n-1})\Big]\times(\mathbb{R}^{n})^{|\mathscr{V}|}$ as $H(g,u)=(H^{1}_{g}(u),H^{2}_{g}(u))$. Then $H$ is of class $C^{k-2}$ if $\mathcal{M}=\mathcal{M}^{k}$ and the previous proposition implies that given $u\in C_{0}$, $u$ is stationary with respect to $g$ if and only if $H(g,u)=0$. Thus if
\begin{equation*}
    \mathcal{S}^{k}_{0}(\Gamma)=\{(g,f)\in\mathcal{M}^{k}\times\hat{\Omega}(\Gamma,M):f\text{ is stationary with respect to }g\}
\end{equation*}
then for any chart $(\hat{U},\Sigma)$ we have $\Sigma(\mathcal{S}^{k}_{0}(\Gamma)\cap\hat U)=H^{-1}(0)$, hence we want to study $H^{-1}(0)$. For technical reasons that will become evident in the subsequent proofs, we will restrict our attention to embedded $\Gamma$-nets (and consider only good* weighted multigraphs as stated at the beginning of the section). Therefore we define
\begin{equation*}
    \mathcal{S}^{k}(\Gamma)=\{(g,f)\in\mathcal{M}\times\hat{\Omega}^{emb}(\Gamma,M):f\text{ is stationary with respect to }g\}\subseteq\mathcal{S}^{k}_{0}(\Gamma)
\end{equation*}
and we assume that all charts $(\hat{U},\Sigma)$ considered verify $\hat{U}\cap\hat{\Omega}(\Gamma,M)\subseteq\hat{\Omega}^{emb}(\Gamma,M)$. We are going to show that under the previous conditions $0$ is a regular value for $H$. In order to do that, we need to study $DH$ which is associated with the Hessian of the length functional. For that purpose, in the remainder of this section we derive the second variation formula, define the notion of Jacobi field and discuss the relation between these definitions in local coordinates and the intrinsic ones given in Section \ref{setup}.

Let $f:(-\varepsilon,\varepsilon)^{2}\to\hat{\Omega}(\Gamma,M)\cap\hat{U}$ be a two parameter family and denote $u_{xs}=u(x,s)=\Sigma(f(x,s))$ the corresponding two parameter family in $C_{0}$. Assume $u_{00}$ is stationary and denote $X(t)=\frac{\partial u}{\partial s}(0,0,t)=(a_{E}^{X},b_{E}^{X},u_{E}^{X}(t))_{E\in\mathscr{E}}$ and $Y(t)=\frac{\partial u}{\partial x}(0,0,t)$. We know that $X,Y\in T_{u_{0}}C_{0}=\ker(DC_{u_{0}})$. Using (\ref{FVF 3}), we have

\begin{align*}
    \frac{\partial^{2}}{\partial x\partial s}\bigg|_{(0,0)}l_{g}(u(x,s)) = & \frac{d}{dx}\bigg|_{x=0}\Bigg[\sum_{E\in\mathscr{E}}\int_{E}H^{1,E}_{g}(u_{x0})(t)\cdot\frac{\partial u_{E}}{\partial s}(x,0,t)dt\\
    & +\sum_{v\in\mathscr{V}}H^{2,v}_{g}(u_{x0})\cdot (\frac{\partial c_{i_{v}}(E_{v})}{\partial s}(x,0),\frac{\partial u_{E_{v}}}{\partial s}(x,0,i_{v})\Bigg]\\
    = &  \sum_{E\in\mathscr{E}}\int_{E}DH^{1,E}_{g}(u_{00})(Y)(t)\cdot u^{X}_{E}(t) dt\\
    & +\sum_{v\in\mathscr{V}}DH^{2,v}_{g}(u_{00})(Y)\cdot ( c_{i_{v}}^{X}(E_{v}),u^{X}_{E_{v}}(i_{v})).
\end{align*}

The Hessian $\Hess_{u_{0}} l_{g}$ of the length functional at a critical point $u_{0}\in C_{0}$ is the continuous bilinear map $\Hess_{u_{0}}l_{g}:T_{u_{0}}C_{0}\times T_{u_{0}}C_{0}\to\mathbb{R}$ given by $\Hess_{u_{0}} l_{g}(X,Y)=\frac{\partial^{2}}{\partial x\partial s}|_{(0,0)}l_{g}(u(x,s))$ where $u(x,s)$ is a two parameter family in $C_{0}$ such that $X(t)=\frac{\partial u}{\partial s}(0,0,t)$ and $Y(t)=\frac{\partial u}{\partial x}(0,0,t)$. The previous computation shows that the Hessian is well defined and
\begin{align*}
    \Hess l_{g}(u_{0})(X,Y)= & \sum_{E\in\mathscr{E}}\int_{E}DH^{1,E}_{g}(u_{00})(Y)(t)\cdot u^{X}_{E}(t) dt\\
    & +\sum_{v\in\mathscr{V}}DH^{2,v}_{g}(u_{00})(Y)\cdot ( c_{i_{v}}^{X}(E_{v}),u^{X}_{E_{v}}(i_{v})).
\end{align*}
A vector field $X\in T_{u_{0}}C_{0}$ is said to be Jacobi along $u_{0}$ if it is a null vector for $\Hess_{u_{0}}l_{g}$ i.e. if for every $Y\in T_{u_{0}}C_{0}$ we have $\Hess l_{g}(u_{0})(X,Y)=0$. Arguing as we did before with the first variation formula, it is clear that $X\in T_{u_{0}}C_{0}$ is Jacobi along $u_{0}$ if and only if $DH^{1}_{g}(u_{0})(X)=DH^{2}_{g}(u_{0})(X)=0$.

\begin{definition}
Given a critical point $u_{0}$ of the length functional $l_{g}$, we say that $u_{0}$ is nondegenerate if the only Jacobi field $X$ along $u_{0}$ is the zero vector field.
\end{definition}

Now we can study the relation between this notion of Jacobi field and nondegeneracy and that presented in Section \ref{setup}. Denote $W_{E}\subseteq\Omega(E,M)$ the image of $\Theta_{E}:W_{3}(f_{1}|_{E})\to\Omega(E,M)$ and $W=\prod_{E\in\mathscr{E}}W_{E}$. Consider the map $\tilde{\Sigma}:W\to U=\prod_{E\in\mathscr{E}}U_{E}$ defined as $\tilde{\Sigma}(g)=(\Xi'(\Theta_{E}^{-1}(g_{E})))_{E\in\mathscr{E}}$. We can establish a correspondence between vector fields $X\in T_{u_{0}}C_{0}$ along $u_{0}$ and vector fields $J$ along $f_{0}=\tilde{\Lambda}(u_{0})$ by $J=D\Tilde{\Lambda}_{u_{0}}(X)$ and $X=D\tilde{\Sigma}_{f_{0}}(J)$. Notice that we can assume $J$ is at least $C^{3}$ as we are working with Jacobi fields along a geodesic net with respect to a $C^{k}$ metric with $k\geq 3$, and therefore $\tilde{\Sigma}:\mathfrak{X}^{3}(f_{0})\to U\subseteq\prod_{E\in\mathscr{E}}\mathbb{R}^{2}\times C^{2}(E,\mathbb{R}^{n-1})$ is differentiable. As $\tilde{\Sigma}$ is not exactly the inverse of $\tilde{\Lambda}$, one would not expect this correspondence to be bijective. However, we have the following characterization if we restrict to the space of embedded $\Gamma$-nets.

\begin{prop}\label{Prop Jacobi}
Let $f_{0}\in\Omega^{emb}(\Gamma,M)$ where $\Gamma$ is a good* weighted multigraph and let $u_{0}\in C_{0}$ be such that $\tilde{\Lambda}(u_{0})=f_{0}$. Let $X\in T_{u_{0}}C_{0}$ and let $J$ be a vector field along $f_{0}$. Assume $f_{0}$ is stationary with respect to a metric $g$. Then
\begin{enumerate}
    \item If $J=D\tilde{\Lambda}_{u_{0}}(X)$ is parallel along $f_{0}$ then $X=J=0$.
    \item  $D\tilde{\Sigma}_{f_{0}}(D\tilde{\Lambda}_{u_{0}}(X))=X$.
    \item $D\tilde{\Lambda}_{u_{0}}(D\tilde{\Sigma}_{f_{0}}(J))=J+K$ where $K$ is a parallel vector field along $f_{0}$.
    \item If $X$ is Jacobi along $u_{0}$ then $J=D\tilde{\Lambda}_{u_{0}}(X)$ is Jacobi along $f_{0}$.
    \item If $J$ is Jacobi along $f_{0}$ then $X=D\tilde{\Sigma}_{f_{0}}(J)$ is Jacobi along $u_{0}$.
\end{enumerate}
\end{prop}

\begin{proof}
First let us show that if $J=D\tilde{\Lambda}_{u_{0}}(X)$ is parallel along $f_{0}$ then $X=J=0$. Given $v\in\mathscr{V}$, as there are at least three different incoming edges at $v$, the tangent lines to two of them at $v$ should be different because otherwise two of them would have the same inward tangent unit vector and therefore by uniqueness of the solutions of the geodesic equation $f_{0}$ would not be embedded, absurd. Therefore $J(v)=0$ for all $v\in\mathscr{V}$ (as $J(v)$ has to be colinear with all the inward tangent vectors to the edges at $v$). If $X=(c_{E},d_{E},w_{E})_{E\in\mathscr{E}}$, $u_{0}=(a_{E},b_{E},u_{0,E})_{E\in\mathscr{E}}$ and $v_{0,E}=\Xi(a_{E},b_{E},u_{0,E})$, then $K_{E}=D\Xi_{u_{0,E}}(X)$ is parallel along $v_{0,E}$ and $K_{E}(0)=K_{E}(1)=0$ for all $E\in\mathscr{E}$ (as $f_{0}|_{E}(t)=\mathcal{E}\circ\phi_{E}(v_{0,E}(t))$ and $J_{E}=D(\mathcal{E}\circ\phi_{E})_{v_{0,E}}(K_{E})$). But we know that $K_{E}(t)=(c_{E}(1-t)+d_{E}t,w_{E}(t))$ thus $c_{E}=d_{E}=0$. Then there exists a $C^{2}$ function $h:[0,1]\to\mathbb{R}$ such that
\begin{equation*}
    K_{E}(t)=(0,w_{E}(t))=h(t)\dot{v}_{0,E}(t)=h(t)(b_{E}-a_{E},\dot{u}_{0,E}(t)).
\end{equation*}
This implies that $h\equiv 0$ and hence $X=J=0$ as claimed in \textit{(1)}.

Now take $X\in T_{u_{0}}C_{0}$ and a one parameter family $u:(-\varepsilon,\varepsilon)\to U$ such that $\frac{d}{ds}|_{s=0}u_{s}=X$. Then
\begin{equation*}
    D\tilde{\Sigma}_{f_{0}}(D\tilde{\Lambda}_{u_{0}}(X))=\frac{d}{ds}\bigg|_{s=0}\tilde{\Sigma}(\tilde{\Lambda}(u_{s}))=\frac{d}{ds}\bigg|_{s=0}u_{s}=X
\end{equation*}
as $\tilde{\Sigma}\circ\tilde{\Lambda}=id_{U}$ because of Lemma \ref{Lemma section}. This proves \textit{(2)}.

On the other hand, if we take $J\in T_{f_{0}}\Omega(\Gamma,M)$ and a one parameter family $f:(-\varepsilon,\varepsilon)\to U$ such that $\frac{d}{ds}|_{s=0}f_{s}=J$, we have that $\tilde{\Lambda}(\tilde{\Sigma}(f_{s}))$ is a reparametrization of $f_{s}$ for all $s\in(-\varepsilon,\varepsilon)$ because of Remark \ref{Rk xi}. Writing $\tilde{\Lambda}(\tilde{\Sigma}(f_{s}))(t)=f(s,\theta(s,t))$ and using that $\theta(0,t)=t$ for every $t\in\Gamma$ as $\tilde{\Lambda}(u_{0})=f_{0}$, we can see
\begin{equation*}
    D\tilde{\Lambda}_{u_{0}}(D\tilde{\Sigma}_{f_{0}}(J))=\frac{d}{ds}\bigg|_{s=0}\tilde{\Lambda}(\tilde{\Sigma}(f_{s}))=\frac{d}{ds}\bigg|_{s=0}f(s,\theta(s,t))=J(t)+\frac{\partial\theta}{\partial s}(0,t)\dot{f}_{0}(t)
\end{equation*}
so we get \textit{(3)} by defining $K(t)=\frac{\partial\theta}{\partial s}(0,t)\dot{f}_{0}(t)$.

Now let $X$ be a Jacobi field along $u_{0}$ (which is assumed to be stationary). Let $J=D\Tilde{\Lambda}_{u_{0}}(X)$ and let $f(-\varepsilon,\varepsilon)^{2}\to\Omega(\Gamma,M)$ be a two parameter family such that $J(t)=\frac{\partial f}{\partial s}(0,0,t)$. Consider the two parameter family $u(x,s)=\tilde{\Sigma}(f(x,s))$ through $u_{0}$. Then
\begin{equation*}
    D\tilde{\Lambda}_{u_{0}}(\frac{\partial u}{\partial s}(0,0))=D\tilde{\Lambda}_{u_{0}}(D\tilde{\Sigma}_{f_{0}}(J))=J+K
\end{equation*}
for some parallel vector field $K$ along $f_{0}$ because of \textit{(3)} which we have just proved. Therefore $K=D\tilde{\Lambda}_{u_{0}}(\frac{\partial u}{\partial s}(0,0)-X)$ and due to \textit{(1)} it must be $K=0$. Therefore $D\tilde{\Lambda}_{u_{0}}(\frac{\partial u}{\partial s}(0,0))=J=D\tilde{\Lambda}_{u_{0}}(X)$ and hence $X=\frac{\partial u}{\partial s}(0,0)$ because $D\Tilde{\Lambda}_{u_{0}}$ is a monomorphism. As $X$ is Jacobi and $\tilde{\Lambda}(u(x,s))$ is a reparametrization of $f(x,s)$ for all $(x,s)\in(-\varepsilon,\varepsilon)^{2}$, this implies that
\begin{equation*}
    \frac{\partial^{2}}{\partial x\partial s}\bigg|_{(0,0)}l_{g}f(x,s)=\frac{\partial^{2}}{\partial x\partial s}\bigg|_{(0,0)}l_{g}u(x,s)=0
\end{equation*}
and hence $J$ is Jacobi along $f_{0}$. This proves \textit{(4)}.

Let $J$ be a Jacobi field along the stationary geodesic net $f_{0}$. Let $X=D\tilde{\Sigma}_{f_{0}}(J)$ and let $u:(-\varepsilon,\varepsilon)^{2}\to U$ be a two parameter family with $u(0,0)=u_{0}$ and $\frac{\partial u}{\partial s}(0,0)=X$. Then if $f(x,s)=\tilde{\Lambda}(u(x,s))$
\begin{equation*}
    \frac{\partial f}{\partial s}(0,0)=D\tilde{\Lambda}_{u_{0}}(X)=D\tilde{\Lambda}_{u_{0}}(D\tilde{\Sigma}_{f_{0}}(J))=J+K
\end{equation*}
for some parallel vector field $K$ along $f_{0}$ because of \textit{(3)}. As both $J$ and $K$ are Jacobi along $f_{0}$, so is $J+K$ and hence
\begin{equation*}
    \frac{\partial^{2}}{\partial x\partial s}\bigg|_{(0,0)}l_{g}u(x,s)=\frac{\partial^{2}}{\partial x\partial s}\bigg|_{(0,0)}l_{g}f(x,s)=0
\end{equation*}
so we can deduce that $X$ is Jacobi along $u_{0}$, which completes the proof of \textit{(5)}.
\end{proof}

From the proposition we can see that given a critical point $u_{0}$ of the length functional with respect to $g$, $u_{0}$ is nondegenerate in the sense that it does not admit any nonzero Jacobi field if and only if $f_{0}=\tilde{\Lambda}(u_{0})$ is nondegenerate in the sense that every Jacobi field is parallel. Hence the two notions of nondegeneracy are equivalent.

\section{$D_{2}H$ is Fredholm of index $0$}\label{fredholm}
Fix a good* weighted multigraph $\Gamma$. Let us continue working in local coordinates $(\hat{U},\Sigma)$ verifying $\hat{U}\cap\hat{\Omega}(\Gamma,M)\subseteq\hat{\Omega}^{emb}(\Gamma,M)$ as we have been doing previously. The goal of this section is to prove the following result.

\begin{prop}\label{Prop Fredholm}
Given $u_{0}\in C_{0}$ and $g\in\mathcal{M}$ such that $H(g,u_{0})=0$, the operator $D_{2}H_{(g,u_{0})}:T_{u_{0}}C_{0}\to \Big[\prod_{E\in\mathscr{E}}C^{0}(E,\mathbb{R}^{n-1})\Big]\times(\mathbb{R}^{n})^{|\mathscr{V}|}$ is Fredholm of index $0$.
\end{prop}

We will need to introduce some notation and use the subsequent two lemmas. Their proofs are elementary using the definitions and results discussed in \cite{Tsoy}.

\begin{definition}
    Given a Riemannian metric $h$ on $M$, a $\Gamma$-net $f\in\Omega(\Gamma,M)$ and an integer $k\geq 0$, we denote
    \begin{align*}
        \mathfrak{X}^{k}_{0}(f) & =\{Z\in\mathfrak{X}^{k}(f):Z(v)=0\text{ }\forall v\in\mathscr{V}\},\\
        \mathfrak{X}^{k}(f)^{\parallel} & =\{Z\in\mathfrak{X}^{k}(f):Z_{E}\text{ is parallel along }f_{E}\text{ }\forall E\in\mathscr{E}\},\\
        \mathfrak{X}^{k}(f)^{\perp}_{h} & =\prod_{E\in\mathscr{E}}\{Z_{E}\in\mathfrak{X}^{k}(f_{E}):Z_{E}\text{ is normal along }f_{E}\text { with respect to }h\},\\
        \mathfrak{X}_{0}^{k}(f)^{\perp}_{h} & =\{Z\in\mathfrak{X}^{k}(f)^{\perp}_{h}:Z_{E}(0)=Z_{E}(1)=0\text{ }\forall E\in\mathscr{E}\}.
    \end{align*}
    Given $t\in\Gamma$ and $v\in T_{f(t)}M$, we denote $v^{\perp}_{h}$ the projection of $v$ to the orthogonal complement of $\langle\dot{f}(t)\rangle$ with respect to $h$.
\end{definition}

\begin{rk}
    Observe that $\mathfrak{X}^{k}(f)^{\perp}_{h}\not\subseteq\mathfrak{X}^{k}(f)$  but $\mathfrak{X}_{0}^{k}(f)^{\perp}_{h}\subseteq\mathfrak{X}_{0}^{k}(f)$.
\end{rk}

\begin{lemma}\label{Lemma F1}
    Let $h_{1}$ and $h_{2}$ be two Riemannian metrics on $M$ and $k\geq 0$. Then the map $O:\mathfrak{X}^{k}(f)^{\perp}_{h_{1}}\to\mathfrak{X}^{k}(f)^{\perp}_{h_{2}}$, $Z\mapsto ((Z_{E})^{\perp}_{h_{2}})_{E\in\mathscr{E}}$ is a continuous linear isomorphism. The same holds changing $\mathfrak{X}^{k}(f)_{h_{i}}^{\perp}$ by $\mathfrak{X}^{k}_{0}(f)_{h_{i}}^{\perp}$.
\end{lemma}

\begin{lemma}\label{Lemma F2}
Let $\mathcal{F},\mathcal{G}$ be Banach spaces and $L:\mathcal{F}\to \mathcal{G}$ be a linear and continuous map. Let $\mathcal{F}_{0},\mathcal{F}_{1}\subseteq \mathcal{F}$ and $\mathcal{G}_{0},\mathcal{G}_{1}\subseteq \mathcal{G}$ be closed subspaces with $\dim(\mathcal{F}_{1})=\dim(\mathcal{G}_{1})<\infty$ such that $\mathcal{F}_{0}\bigoplus\mathcal{F}_{1}=\mathcal{F}$ and $\mathcal{G}_{0}\bigoplus\mathcal{G}_{1}=\mathcal{G}$. Let $L_{ij}:\mathcal{F}_{i}\to\mathcal{G}_{j}$, $i,j\in\{0,1\}$ be such that $L(f_{0},f_{1})=(L_{00}(f_{0})+L_{10}(f_{1}),L_{01}(f_{0})+L_{11}(f_{1}))$ for each $f_{0}\in\mathcal{F}_{0}$ and $f_{1}\in\mathcal{F}_{1}$. Assume $L_{11}:\mathcal{F}_{0}\to \mathcal{G}_{0}$ is Fredholm of index $0$. Then $L:\mathcal{F}\to \mathcal{G}$ is Fredholm of index $0$.
\end{lemma}

\begin{proof}[Proof of Proposition \ref{Prop Fredholm}]
Let $g\in\mathcal{M}$ and $u_{0}\in C_{0}$ verifying $H(g,u_{0})=0$. Consider the spaces

\begin{align*}
    C^{2}_{0}(E,\mathbb{R}^{n-1}) & =\{u\in C^{2}(E,\mathbb{R}^{n-1}):u(0)=u(1)=0\},\\
    \mathcal{F} & = T_{u_{0}}C_{0},\\
    \mathcal{F}_{0} & =\prod_{E\in\mathscr{E}}\{0\}\times C^{2}_{0}(E,\mathbb{R}^{n-1}),\\
    \mathcal{G} & =\big[\prod_{E\in\mathscr{E}}C^{0}(E,\mathbb{R}^{n-1})\big]\times(\mathbb{R}^{n})^{|\mathscr{V}|},\\
    \mathcal{G}_{0} & =\big[\prod_{E\in\mathscr{E}}C^{0}(E,\mathbb{R}^{n-1}) \big]\times\{0\},\\
    \mathcal{G}_{1} & =\{0\}\times(\mathbb{R}^{n})^{|\mathscr{V}|}.
\end{align*}
Observe that as $\mathcal{F}_{0}\subseteq\mathcal{F}\subseteq\mathcal{B}$,
\begin{equation*}
    \codim_{\mathcal{F}}(\mathcal{F}_{0})=\codim_{\mathcal{B}}(\mathcal{F}_{0})-\codim_{\mathcal{B}}(\mathcal{F})=2n|\mathscr{E}|-n(2|\mathscr{E}|-|\mathscr{V}|)=n|\mathscr{V}|
\end{equation*}
and hence there exists a subspace $\mathcal{F}_{1}\subseteq\mathcal{F}$ of dimension $n|\mathscr{V}|$ such that $\mathcal{F}_{0}\bigoplus\mathcal{F}_{1}=\mathcal{F}.$ As $\mathcal{G}_{0}\bigoplus\mathcal{G}_{1}=\mathcal{G}$ and $\dim(\mathcal{G}_{1})=n|\mathscr{V}|=\dim(\mathcal{F}_{1})$, by Lemma \ref{Lemma F2} it suffices to show that $D_{2}H^{1}_{(g,u_{0})}:\mathcal{F}_{0}\to \mathcal{G}_{0}$ is Fredholm of index $0$.

Let $X,Y\in \mathcal{F}_{0}$ with $X=(0,0,u_{E})_{E\in\mathscr{E}}$ and $Y=(0,0,v_{E})_{E\in\mathscr{E}}$, $u_{E},v_{E}\in C^{2}_{0}(E,\mathbb{R}^{n-1})$ $\forall E\in\mathscr{E}$. By the second variation formula, we have
\begin{equation}\label{FredEq1}
    \frac{\partial^{2}}{\partial x\partial s}\bigg|_{(0,0)}l_{g}(u_{xs})=\sum_{E\in\mathscr{E}}\int_{E}D_{2}H^{1,E}_{(g,u_{0})}(Y)(t)\cdot u_{E}(t)dt
\end{equation}
where $u_{x,s}=u_{0}+sX+xY$. On the other hand, if $f(x,s)=\Tilde{\Lambda}(u_{xs})$, then as the vector fields $\Tilde{X}=D\Tilde{\Lambda}_{u_{0}}(X)$ and $\Tilde{Y}=D\Tilde{\Lambda}_{u_{0}}(Y)$ vanish at the vertices of $\Gamma$ we have
\begin{equation}\label{FredEq2}
     \frac{\partial^{2}}{\partial x\partial s}\bigg|_{(0,0)}l_{g}(f_{xs})=\sum_{E\in\mathscr{E}}-\frac{n(E)}{l(E)}\int_{E}g(\ddot{\Tilde{Y}}^{\perp}_{g}+R(\dot{f}_{0},\Tilde{Y}^{\perp}_{g})\dot{f}_{0},\Tilde{X}(t)^{\perp}_{g})dt
\end{equation}
being $f_{0}=f_{00}$. Notice that as $f_{0}=\Tilde{\Lambda}(u_{0})$ is embedded with domain a good* weighted multigraph, $\mathfrak{X}^{2}(f_{0})^{\parallel}\subseteq\mathfrak{X}^{2}_{0}(f_{0})$. Observe that $D\tilde{\Lambda}_{u_{0}}(\mathcal{F}_{0})=\mathfrak{X}^{2}_{0}(f_{0})^{\perp}_{\gamma_{0}}$, where $\gamma_{0}$ is the the auxiliary Riemannian metric used to construct the atlas of $\hat{\Omega}(\Gamma,M)$ in Section \ref{lengthsgn}. Then as $D\Tilde{\Lambda}_{u_{0}}$ is a monomorphism, by the Open Mapping Theorem it is an isomorphism between the spaces $\mathcal{F}_{0}$ and $\mathfrak{X}^{2}_{0}(f_{0})^{\perp}_{\gamma_{0}}$. Similarly, $D\Tilde{\Lambda}_{u_{0}}:\mathcal{G}_{0}\to\mathfrak{X}^{0}(f_{0})^{\perp}_{\gamma_{0}}$ is a linear isomorphism. Let us define the operators $M:\mathcal{F}_{0}\to\mathfrak{X}^{2}_{0}(f_{0})^{\perp}_{g}$, $M=O\circ D\Tilde{\Lambda}_{u_{0}}$ and $N:\mathcal{G}_{0}\to\mathfrak{X}^{0}(f_{0})^{\perp}_{g}$, $N=O\circ D\Tilde{\Lambda}_{u_{0}}$ ($O$ is as in Lemma \ref{Lemma F1} with respect to the metrics $h_{1}=\gamma_{0}$ and $h_{2}=g$). By Lemma \ref{Lemma F1}, $M$ and $N$ are isomorphisms.

Let $L:\mathfrak{X}^{2}(f_{0})\to\mathfrak{X}^{0}(f_{0})^{\perp}$ be the operator
\begin{equation*}
L(Z)=-(\frac{n(E)}{l(E)}\bigg[(\ddot{Z}_{E})_{g}^{\perp}+R(\dot{f}_{0},(Z_{E})_{g}^{\perp})\dot{f}_{0}\bigg])_{E\in\mathscr{E}}.
\end{equation*}
It holds $L(O(Z))=L(Z)$ for every $Z\in\mathfrak{X}^{2}(f_{0})$ (as $L$ vanishes over parallel vector fields). In particular, $L(\Tilde{Y})=L(O(\Tilde{Y}))=L(M(Y))$. Also $L:\mathfrak{X}^{2}_{0}(f_{0})^{\perp}_{g}\to\mathfrak{X}^{0}(f_{0})^{\perp}_{g}$ is Fredholm of index $0$, as so are the Jacobi operators $L_{E}:\mathfrak{X}^{2}_{0}(f_{0,E})^{\perp}_{g}\to\mathfrak{X}^{0}(f_{0,E})^{\perp}_{g}$, $L_{E}(Z)=\ddot{Z}_{g}+R(\dot{f}_{0},Z)\dot{f}_{0}$ (because they are elliptic). Therefore, by (\ref{FredEq1}), (\ref{FredEq2}) and the fact that $l_{g}(f_{xs})=l_{g}(u_{xs})$ we get
\begin{equation}\label{FredEq3}
    \sum_{E\in\mathscr{E}}\int_{E}D_{2}H^{1,E}_{(g,u_{0})}(Y)(t)\cdot u_{E}(t)dt=\sum_{E\in\mathscr{E}}\int_{E}g(L(\Tilde{Y})(t),\Tilde{X}(t)^{\perp}_{g})dt.
\end{equation}
Let $\langle,\rangle_{1}$ and $\langle,\rangle_{2}$ denote the inner products in $\mathfrak{X}^{0}(f_{0})^{\perp}_{g}$ and $\mathcal{G}_{0}$ respectively given by
\begin{align*}
    \langle Z_{1},Z_{2}\rangle_{1} & =\sum_{E\in\mathscr{E}}\int_{E}g(Z_{1}(t),Z_{2}(t))dt,\\
    \langle X_{1},X_{2}\rangle_{2} & =\sum_{E\in\mathscr{E}}\int_{E}u_{1}(t)\cdot u_{2}(t)dt
\end{align*}
where $X_{i}=(0,0,u^{i}_{E})_{E\in\mathscr{E}}$, $i=1,2$. Let $N^{*}$ denote the adjoint of $N$ with respect to these inner products, i.e. the map $N^{*}:\mathfrak{X}^{0}(f_{0})^{\perp}_{g}\to\mathcal{G}_{0}$ such that for every $X\in\mathcal{G}_{0}$ and every $Z\in\mathfrak{X}^{0}(f_{0})^{\perp}_{g}$ it holds
\begin{equation*}
    \langle Z,N(X)\rangle_{1}=\langle N^{*}(Z),X\rangle_{2}.
\end{equation*}
Then
\begin{align*}
    \sum_{E\in\mathscr{E}}\int_{E}g(L(\Tilde{Y})(t),\Tilde{X}(t)^{\perp})dt & = \langle L(\Tilde{Y}),O(\Tilde{X})\rangle_{1}\\
     & =\langle L(M(Y)),N(X)\rangle_{1}\\
    & =\langle N^{*}(L(M(Y))),X\rangle_{2}
\end{align*}
and by (\ref{FredEq3}) we deduce that
\begin{equation*}
    \langle D_{2}H^{1}_{(g,u_{0})}(Y),X\rangle_{2}=\langle N^{*}(L(M(Y))),X\rangle_{2}
\end{equation*}
holds for every $X,Y\in\mathcal{F}_{0}$. Hence, $D_{2}H^{1}_{(g,u_{0})}=N^{*}\circ L\circ M$ and as both $M,N^{*}$ are linear isomorphisms this yields that $D_{2}H^{1}_{(g,u_{0})}:\mathcal{F}_{0}\to\mathcal{G}_{0}$ is Fredholm of index $0$, as desired.
\end{proof}

We finish this section by proving the following lemma which will be used in \cite{Liokumovich}.

\begin{lemma} \label{nondegenerate}
 Let $\Gamma$ be a good* weighted multigraph and  $f_{0}: \Gamma \rightarrow M$ be an embedded non-degenerate stationary geodesic net with respect to a $C^{k}$ metric $g_{0}$, $k\geq 3$. Then there exists a neighborhood $W$ of $g_{0}$ in $\mathcal{M}^{k}$ and a differentiable map $\Delta:W\to\Omega(\Gamma,M)$ such that $\Delta(g)$ is a non-degenerate stationary geodesic net with respect to $g$ for every $g\in W$.
\end{lemma}

\begin{proof}
Let $f_{0}:\Gamma\to M$ be as in the statement of the lemma. Take a chart $(\hat{U},\Sigma)$ of $\hat{\Omega}(\mathscr{E},M)$ containing $[f_{0}]$ as constructed in the previous section. Then we know that there exists a differentiable map $H:\mathcal{M}\times C_{0}\to\mathcal{Y}$ such that $[f]\in \hat{U}\cap\hat{\Omega}(\Gamma,M)$ is stationary with respect to $g$ if and only if $H(g,\Sigma([f]))=0$. As $f_{0}$ is nondegenerate and embedded, $u_{0}=\Sigma([f_{0}])$ is nondegenerate and hence we know that $D_{2}H(g_{0},u_{0})$ is an isomorphism (here we are using that $D_{2}H(g_{0},u_{0})$ is Fredholm of index $0$ and Proposition \ref{Prop Jacobi}). Applying the Implicit Function Theorem to the map $H$ at the point $(g_{0},u_{0})$, we get that there is a neighborhood $W$ of $g_{0}$ in $\mathcal{M}^{k}$ and a differentiable map $\Delta:W\to\Omega(\Gamma,M)$ with $\Delta(g_{0})=f_{0}$ such that $\Delta(g)$ is stationary with respect to $g$ for all $g\in W$. By continuity of the Hessian with respect to the metric, shrinking $W$ if necessary we can guarantee that $\Delta(g)$ is nondegenerate.
\end{proof}

\section{Proofs of the Structure Theorem and of the Bumpy Metrics Theorem in the case $k<\infty$}\label{banachmanstr}

Let $\Gamma$ be a good* weighted multigraph. Recall that
\begin{equation*}
    \mathcal{S}^{k}(\Gamma)=\{(g,f)\in\mathcal{M}^{k}\times\hat{\Omega}^{emb}(\Gamma,M):f\text{ is stationary with respect to }g\}.
\end{equation*}


We are going to prove that given a chart $(\hat{U},\Sigma)$ as described in the previous sections, $id\times\Sigma(\mathcal{S}^{k}(\Gamma)\cap \mathcal{M}^{k}\times\hat{U})$ is a $C^{k-2}$ Banach submanifold of $\mathcal{M}^{k}\times C_{0}$. Then we will use this to construct an atlas for $\mathcal{S}^{k}(\Gamma)$. We know that
\begin{equation*}
    id\times\Sigma(\mathcal{S}^{k}(\Gamma)\cap(\mathcal{M}^{k}\times\hat{U}))=\{(g,u)\in\mathcal{M}^{k}\times C_{0}:H(g,u)=0\}=H^{-1}(0)
\end{equation*}
so our strategy will be to prove that $0$ is a regular value of $H$. For that purpose we will need \cite[Theorem~1.2]{White}, which we state below.

\begin{thm}\label{MST}
Let $\mathcal{M}$, $X$ and $Y$ be Banach spaces and $\mathcal{H}$ be a Hilbert Space with $X\subseteq Y\subseteq\mathcal{H}$. Let $L:\mathcal{M}\times X\to\mathbb{R}$ be a $C^{2}$ function and suppose there is a $C^{q}$ map $H:\mathcal{M}\times X\to Y$ such that
$$\frac{d}{dt}\bigg|_{t=0}L(g,u+tv)=\langle H(g,u),v\rangle$$
for all $g\in\mathcal{M}$ and all $u,v\in X$. Suppose also that $D_{2}H(g_{0},u_{0}):X\to Y$ is a Fredholm map of Fredholm index $0$ and that for every nonzero $\kappa\in K=\ker(D_{2}H(g_{0},u_{0}))$ there exists a one parameter family $g(s)\in\mathcal{M}$ with $g(0)=g_{0}$ such that
\begin{equation}\label{eqn:C}
    \frac{\partial^{2}}{\partial s\partial t}\bigg|_{s=t=0}L(g(s),u_{0}+t\kappa)\neq 0\tag{C}.
\end{equation}
Then:
\begin{enumerate}
    \item The map $H:\mathcal{M}\times X\to Y$ is a submersion near $(g_{0},u_{0})$, so there exists a neighborhood $W$ of $(g_{0},u_{0})$ such that
    $$\mathcal{S}=\{(g,u)\in W:H(g,u)=0\}$$
    is a $C^{q}$ Banach submanifold of $\mathcal{M}\times X$ and
    $$T_{(g,u)}\mathcal{S}=\ker(DH_{(g,u)})$$
    for all $(g,u)\in\mathcal{S}$.
    \item The projection $\Pi:\mathcal{S}\to\mathcal{M}$, $\Pi(g,u)=g$ is a $C^{q}$ Fredholm map of index $0$.
\end{enumerate}
\end{thm}

We want to apply the previous theorem for $\mathcal{M}=\mathcal{M}^{k}$, $X=\Sigma(\hat{U}\cap\hat{\Omega}(\Gamma,M))=C_{0}$ which is a Banach submanifold of $U=\Sigma(\hat{U})$ modeled in the Banach space $\mathcal{X}_{0}=\ker (DC_{u_{0}})\subseteq\mathcal{B}=\prod_{E\in\mathscr{E}}\mathbb{R}^{2}\times C^{2}(E,\mathbb{R}^{n-1})$ (notice that the theorem still works if $X$ is assumed to be a Banach manifold instead of a Banach space, as we are focusing on local structure), $Y=\mathcal{Y}=[\prod_{E\in\mathscr{E}}C^{0}(E,\mathbb{R}^{n-1})]\times(\mathbb{R}^{n})^{|\mathscr{V}|}$, $\mathcal{H}=[\prod_{E\in\mathscr{E}}L^{2}(E,\mathbb{R}^{n-1})]\times(\mathbb{R}^{n})^{|\mathscr{V}|}$, where the inner product of two elements $u_{i}=((u_{i,E})_{E\in\mathscr{E}},(a_{i,v})_{v\in\mathscr{V}})$ of \linebreak $[\prod_{E\in\mathscr{E}}L^{2}(E,\mathbb{R}^{n-1})]\times(\mathbb{R}^{n})^{|\mathscr{V}|}$ ($i=1,2$) is given by
\begin{equation*}
    \langle u_{1},u_{2}\rangle=\sum_{E\in\mathscr{E}}\int_{E}u_{1,E}(t)\cdot u_{2,E}(t)dt+\sum_{v\in\mathscr{V}}a_{1,v}\cdot a_{2,v}
\end{equation*}
where $\cdot$ denotes the Euclidean inner product in $\mathbb{R}^{n}$ or $\mathbb{R}^{n-1}$. We consider the inclusion $\iota:\mathcal{X}_{0}\to \mathcal{Y}$ given by $\iota(u)=((u_{E})_{E\in\mathscr{E}},(c_{i_{v}}(E_{v}),u_{E_{v}}(i_{v}))_{v\in\mathscr{V}})$. The map $L:\mathcal{M}\times X\to\mathbb{R}$ is given by $L(g,u)=l_{g}(u)$ and $H:\mathcal{M}\times X\to Y$ is the previously defined map. Notice that $H$ is of class $C^{q}$ for $q=k-2\geq 1$. By the first variation formula, we have
\begin{equation*}
    \frac{d}{dt}\bigg|_{t=0}L(g,u+tv)=\langle H(g,u),\iota(v)\rangle.
\end{equation*}
As we proved in the previous section that $D_{2}H_{(g_{0},u_{0})}$ is Fredholm of index $0$, in order to apply the theorem it only remains to show that condition (C) holds.

\begin{proof}[Proof that Condition (C) holds]
Let us take $\kappa\in \ker(D_{2}H(g_{0},u_{0}))\setminus\{0\}$. Let $u:(-\alpha,\alpha)\to U$ be a one parameter family in $C_{0}$ with $u(0,\cdot)=u_{0}$ and $\frac{d}{ds}|_{s=0}u_{s}=\kappa$. Write $u_{s}=(a_{E}(s),b_{E}(s),u_{s,E})_{E\in\mathscr{E}}$. Consider the corresponding one parameter family $f_{s}=\tilde{\Lambda}(u_{s})$ and the associated vector field $J=\frac{d}{ds}|_{s=0}f_{s}=D\tilde{\Lambda}(\kappa)$. By Proposition \ref{Prop Jacobi} we know that $J$ is Jacobi along $f_{0}$. We want to construct a one parameter family $g(x)$ of metrics with $g(0)=g_{0}$ such that

\begin{equation*}
    \frac{\partial^{2}}{\partial x\partial s}\bigg|_{x=s=0}L(g(x),u_{s})\neq 0.
\end{equation*}
By definition of $L(g,u)$, this is the same as finding $g(x)$ such that

\begin{equation*}
    \frac{\partial^{2}}{\partial x\partial s}\bigg|_{x=s=0}\sum_{E\in\mathscr{E}}\int_{E}L^{E}_{g(x)}(t,a_{E}(s),b_{E}(s),u_{s,E}(t),\Dot{u}_{s,E}(t))dt\neq 0.
\end{equation*}
We will follow the reasoning from \cite{White}. Consider a one parameter family of metrics $g_{x}(z)=(1+xh(z))g_{0}(z)$ conformal to $g_{0}$ (here $h:M\to\mathbb{R}$ is a smooth function). Then given $u=(a_{E},b_{E},u_{E})_{E\in\mathscr{E}}$ and $E\in\mathscr{E}$,
\begin{equation*}
    L^{E}_{g(x)}(t,a_{E},b_{E},u_{E}(t),\dot{u}_{E}(t))=\sqrt{1+xh(f_{E}(t))}L_{g_{0}}(t,a_{E},b_{E},u_{E}(t),\dot{u}_{E}(t))
\end{equation*}
where $f_{E}(t)=\mathcal{E}\circ\phi_{E}(a_{E}(1-t)+b_{E}t,u_{E}(t))$. Suppose that $h$ vanishes along $f_{0}$. Denote $f_{E}(s,t)=\mathcal{E}\circ\phi_{E}(a_{E}(s)(1-t)+b_{E}(s)t,u_{s,E}(t))$ the restriction of the previously defined $f_{s}$ to the edge $E$. Then

\begin{align*}
  & \frac{\partial^{2}}{\partial x\partial s}\bigg|_{x=s=0}\sum_{E\in\mathscr{E}}\int_{E}L^{E}_{g(x)}(t,a_{E}(s),b_{E}(s),u_{s,E}(t),\Dot{u}_{s,E}(t))dt & \notag \\
  &= \frac{d}{ds}\bigg|_{s=0}\sum_{E\in\mathscr{E}}\int_{E}\frac{1}{2}h(f_{E}(s,t))L^{E}_{g_{0}}(t,a_{E}(s),b_{E}(s),u_{s}(t),\Dot{u}_{s}(t))dt\\
  \begin{split}
  &= \frac{1}{2}\sum_{E\in\mathscr{E}}\int_{E}\big[\frac{d}{ds}\bigg|_{s=0}h(f_{E}(s,t))\big]L_{g_{0}}^{E}(t,a_{E}(0),b_{E}(0),u_{0,E}(t),\dot{u}_{0,E}(t))dt\\
  & \quad+ \frac{1}{2}\sum_{E\in\mathscr{E}}\int_{E}h(f_{0}(t))\big[\frac{d}{ds}\bigg|_{s=0}L^{E}_{g_{0}}(t,a_{E}(s),b_{E}(s),u_{s}(t),\Dot{u}_{s}(t))\big]dt
  \end{split}\\
  &= \frac{1}{2}\sum_{E\in\mathscr{E}}\int_{E}\langle\nabla h_{f_{0}(t)},J(t)\rangle_{\gamma_{0}} L_{g_{0}}^{E}(t,a_{E}(0),b_{E}(0),u_{0,E}(t),\dot{u}_{0,E}(t))dt
\end{align*}
where we used that $h(f_{0}(t))=0$ for all $t\in\Gamma$ because $h$ vanishes along $f_{0}$ and that $\frac{\partial f}{\partial s}(0,t)=D\tilde{\Lambda}_{u_{0}}(\kappa)=J$.

By Proposition \ref{Prop Jacobi}, $J=D\tilde{\Lambda}_{u_{0}}(\kappa)$ is not a parallel Jacobi field along $f_{0}=\Lambda(u_{0})$ because $\kappa\neq 0$. Therefore, there must exist an edge $E_{0}\in\mathscr{E}$ and an interior point $t_{0}\in \interior(E_{0})$ such that $J(t_{0})$ is not parallel to $\dot{f}_{0}(t_{0})$. It is possible to define the smooth function $h:M\to\mathbb{R}$ with  the following properties:
\begin{enumerate}
    \item $h$ has support in a small ball around $f_{0}(t_{0})$ which does not intersect $f_{0}(E)$ for any $E\in\mathscr{E}\setminus\{E_{0}\}$.
    \item $h(f_{0}(t))=0$ for all $t\in\Gamma$.
    \item $\langle \nabla h_{f_{0}(t)},J(t)\rangle_{\gamma_{0}}\geq 0$ for all $t\in\Gamma$ and $\langle \nabla h_{f_{0}(t_{0})},J(t_{0})\rangle_{\gamma_{0}}>0$.
\end{enumerate}
As $L_{g_{0}}(t,a_{E}(0),b_{E}(0),u_{0,E}(t),\dot{u}_{0,E}(t))>0$ for all $t\in\Gamma$ we deduce

\begin{equation*}
    \sum_{E\in\mathscr{E}}\int_{E}\langle\nabla h_{f_{0}(t)},J(t)\rangle_{\gamma_{0}}L_{g_{0}}(t,a_{E}(0),b_{E}(0),u_{0,E}(t),\dot{u}_{0,E}(t))dt > 0
\end{equation*}
and hence condition (C) is satisfied and Theorem \ref{MST} can be applied.
\end{proof}

The previous shows that $\mathcal{S}^{k}(\Gamma)\subseteq\mathcal{M}^{k}\times\hat{\Omega}^{emb}(\Gamma,M)$ is a $C^{0}$ embedded Banach submanifold (recall that the charts of $\hat{\Omega}(\Gamma,M)$ do not have differentiable transition maps). We will prove that in fact the transition maps on the induced atlas for $\mathcal{S}^{k}(\Gamma)$ are of class $C^{k-2}$ (i.e. as regular as they can be), using an argument of Brian White (see \cite[p. 179]{White}).

The idea will be to use that $\Pi:\mathcal{S}\to\mathcal{M}^{k}$ is Fredholm of index $0$ (as shown above) to prove that the Banach space modelling $\mathcal{S}^{k}(\Gamma)$ is the one which models $\mathcal{M}^{k}$. This is clear when $\ker(D\Pi_{(g_{0},u_{0})})=0$ by the Inverse Function Theorem. In general, we can do the following construction.

Let $(g_{0},f_{0})\in\mathcal{S}^{k}(\Gamma)$ and consider two charts $(\hat{U}_{1},\Sigma_{1})$ and $(\hat{U}_{2},\Sigma_{2})$ containing $f_{0}$. Denote $\mathcal{S}_{i}=id\times\Sigma_{i}((\mathcal{M}^{k}\times\hat{U}_{1}\cap\hat{U}_{2})\cap \mathcal{S}^{k}(\Gamma))\subseteq\mathcal{M}^{k}\times\mathcal{B}$ and $u_{i}=\Sigma_{i}(f_{0})$. We know that $K_{i}=\ker(D\Pi_{(g_{0},u_{i})})$ is finite dimensional. Define a map $\Psi:\mathcal{S}^{k}(\Gamma)\to\mathcal{M}^{k}\times\mathbb{R}^{Q}$ as
\begin{equation*}
    \Psi(g,f)=(g,\int_{f}\omega_{1},...,\int_{f}\omega_{Q})
\end{equation*}
where $\omega_{1},...,\omega_{Q}$ are smooth $1$-forms on $M$ to be chosen. The idea is to choose these differential forms so that the maps $\Psi\circ\Lambda_{i}:\mathcal{S}_{i}\to\mathcal{M}^{k}\times\mathbb{R}^{Q}$ are $C^{k-2}$ embeddings in a small neighborhood of $u_{i}$. Notice that for any choice of the $\omega_{j}$'s, $\Psi\circ\Lambda_{i}$ is $C^{k-2}$ and has finite dimensional kernel and cokernel, hence it suffices to do the choices so that their differentials are injective. Denote $F^{\omega}(g,u)=\int_{\Lambda_{1}(u)}\omega$, $F^{\omega}:\mathcal{S}_{1}\to\mathbb{R}$. We will choose $\{\omega_{1},...,\omega_{r}\}$ so that for every $\kappa\in K_{1}\setminus\{0\}$ there exists $1\leq j\leq r$ so that $DF^{\omega_{j}}_{(g_{0},u_{1})}(\kappa)\neq 0$. Then we will do the same for $K_{2}$ by choosing the forms $\omega_{r+1},...,\omega_{Q}$ in an analogous way. This will guaranty that $D(\Psi\circ\Lambda_{i})_{(g_{0},u_{i})}$ is a monomorphism for $i=1,2$.

Given $\kappa\in K_{1}\setminus\{0\}$, if $\{\phi_{s}\}_{s}$ is the one-parameter family of diffeomorphisms generated by $J=D\Lambda_{u_{0}}(\kappa)$, we have
\begin{align*}
    DF^{\omega}_{(g_{0},u_{1})}(\kappa) & =\frac{d}{ds}\bigg|_{s=0}\int_{\Lambda(u_{1}+s\kappa_{i})}\omega\\
    & =\frac{d}{ds}\bigg|_{s=0}\int_{f_{0}}\phi_{s}^{*}\omega\\
    & =\int_{f_{0}}\frac{d}{ds}\bigg|_{s=0}\phi_{s}^{*}\omega\\
    & =\int_{f_{0}}\mathcal{L}_{J}\omega\\
    & =\int_{f_{0}}d\iota_{J}\omega+\iota_{J}d\omega
\end{align*}
by Cartan's magic formula. Thus given $\kappa\in K_{1}\setminus\{0\}$, as $J$ is a nonparallel Jacobi field along $f_{0}$, we can pick a point $t_{0}$ in the interior of an edge such that $J(t_{0})\notin\langle \dot{f}_{0}(t_{0})\rangle$. Then we can define a smooth function $h$ with support in a small ball around $f_{0}(t_{0})$, which vanishes along $f_{0}$ and such that $\langle J(t),\nabla h_{f_{0}(t)}\rangle_{\gamma_{0}}\geq 0$ for all $t$, with strict inequality at $t=t_{0}$ as described above. Then if $\eta$ is a $1$-form on $M$ extending $dL_{g}|_{f_{E}}$, the form $\omega=h\eta$ will work. This is because $\int_{f_{0}}d\iota_{J}\omega$ vanishes as $\iota_{J}\omega$ vanishes at all the vertices of $\Gamma$, and as $h$ vanishes along $f_{0}$,
\begin{equation*}
    \int_{f_{0}}\iota_{J}d\omega=\int_{\Gamma} dh\wedge \eta(J(t),\dot{f}_{0}(t))dt=\int_{\Gamma}dh(J(t))\eta(\dot{f}_{0}(t))dt=\int_{\Gamma}\langle\nabla h_{f_{0}(t)},J(t)\rangle_{\gamma_{0}} dt
\end{equation*}
because $dh(\dot{f}_{0}(t))=0$. The previous quantity is strictly positive by construction of $h$, being $DF^{\omega}_{(g_{0},u_{1})}(\kappa)\neq 0$. Although this $\omega$ is a priori only $C^{k}$, we can perturb it slightly so that it becomes smooth but still verifies $DF^{\omega}_{(g_{0},u_{1})}(\kappa)\neq 0$.

On the other hand, given a smooth $1$-form $\omega$, the set $A^{\omega}:=\{\kappa\in K_{1}:DF^{\omega}_{(g_{0},u_{1})}(\kappa)\neq 0\}$ is open. Therefore, $\{A^{\omega}:\omega\text{ 1-form on }M\}$ is an open cover of $K_{1}\setminus\{0\}$. Take $\{\omega_{1},...,\omega_{r}\}$ such that $\{A^{\omega_{j}}:1\leq j\leq r\}$ is a finite subcover of $\{\kappa\in K_{1}:|\kappa|=1\}$ (which is compact independently of the norm $|\cdot|$ we choose for $K_{1}$ as it is finite dimensional). It follows that if $F=(F^{\omega_{1}},...,F^{\omega_{r}})$ then $DF_{(g_{0},u_{1})}(\kappa)\neq 0$ for every $\kappa\in K_{1}\setminus\{0\}$. 
Proceeding equally for $K_{2}$, we obtain that $\Psi\circ\Lambda_{i}$ is an immersion near $(g_{0},u_{i})$ for $i=1,2$. Then the transition map $\Lambda_{2}^{-1}\circ\Lambda_{1}=(\Psi\circ\Lambda_{2})^{-1}\circ(\Psi\circ\Lambda_{1})$ is a $C^{k-2}$ diffeomorphism of Banach manifolds near $(g_{0},u_{1})$, as desired.

Observe that Proposition \ref{Prop Jacobi} implies that if $(g,[f])\in\mathcal{S}^{k}(\Gamma)$, $f$ is nondegenrerate if and only if given a chart $(\hat{U},\Sigma)$ containing $f$ with $\Sigma([f])=u$ we have that $u$ is nondegenerate as defined in Section \ref{lengthsgn}. But the following $4$ conditions are equivalent:
\begin{enumerate}
    \item $D\Pi_{(g,u)}$ is an epimorphism.
    \item $D\Pi_{(g,u)}$ is injective.
    \item $D_{2}H_{(g,u)}$ is injective.
    \item $u$ is nondegenerate with respect to $g$
\end{enumerate}
as $\ker(D\Pi_{(g,u)})=\{0\}\times\ker D_{2}H_{(g,u)}$. This completes the proof of Theorem \ref{structurethm} for good* weighted multigraphs. The theorem for closed loops with multiplicity is a particular case of the Structure Theorem of Brian White proved in \cite{White}, and this covers all the cases of Theorem \ref{structurethm}.

On the other hand, by Smale's version of Sard's Theorem for Banach spaces proved in \cite{Smale}, for each good weighted multigraph $\Gamma$ the subset $\mathcal{N}^{k}(\Gamma)\subseteq\mathcal{M}^{k}$ of regular values of $\Pi:\mathcal{S}^{k}(\Gamma)\to\mathcal{M}^{k}$ is generic in the Baire sense. Observe that parts \textit{(2)} and \textit{(3)} of Theorem \ref{structurethm} imply that $g\in\mathcal{N}^{k}(\Gamma)$ if and only if $g$ is bumpy with respect to $\Gamma$. Considering that the collection $\{\Gamma:\Gamma\text{ is a good weighed multigraph}\}$ is countable, $\mathcal{N}^{k}:=\bigcap_{\Gamma}\mathcal{N}^{k}(\Gamma)$ is also generic in the Baire sense and is by definition the set of bumpy $C^{k}$ metrics. This proves Theorem \ref{bumpythm} in the case $k<\infty$.

\section{$C^{\infty}$ case}\label{Cinfty}
In this section we are going to discuss how to extend Theorem \ref{bumpythm} to $C^{\infty}$ Riemannian metrics (the analog result for minimal submanifolds is stated in \cite{White2}). Denote $\mathcal{M}^{\infty}=\cap_{k\in\mathbb{N}}\mathcal{M}^{k}$ the space of $C^{\infty}$ Riemannian metrics on $M$ equipped with the $C^{\infty}$ topology, which admits a natural Frechet manifold structure.

\begin{thm}\label{Cinfty thm}
The subset $\mathcal{N}^{\infty}\subseteq\mathcal{M}^{\infty}$ of bumpy $C^{\infty}$ metrics is generic in the Baire sense with respect to the $C^{\infty}$ topology.
\end{thm}

In order to prove the theorem, we will need the following lemma.

\begin{lemma}
Let $\mathcal{N}^{k}\subseteq\mathcal{M}^{k}$ be a generic subset in the Baire sense with respect to the $C^{k}$ topology for each $k\in\mathbb{N}_{\geq 3}$. Assume that if $k'\geq k$ then $\mathcal{N}^{k'}=\mathcal{N}^{k}\cap\mathcal{M}^{k'}$. Then $\mathcal{N}^{\infty}=\cap_{k\in\mathbb{N}}\mathcal{N}^{k}\subseteq\mathcal{M}^{\infty}$ is generic in the Baire sense with respect to the $C^{\infty}$ topology.
\end{lemma}

\begin{proof}
Let us write $\mathcal{N}^{3}=\cap_{l\in\mathbb{N}}\mathcal{N}^{3,l}$ where each $\mathcal{N}^{3,l}$ is open and dense in $\mathcal{M}^{3}$ with the $C^{3}$ topology. For each $k\geq 3$ define $\mathcal{N}^{k,l}=\mathcal{N}^{3,l}\cap\mathcal{M}^{k}$. Observe that given $k\geq 3$,
\begin{equation*}
    \bigcap_{l\in\mathbb{N}}\mathcal{N}^{k,l}=(\bigcap_{l\in\mathbb{N}}\mathcal{N}^{3,l})\cap\mathcal{M}^{k}=\mathcal{N}^{3}\cap\mathcal{M}^{k}=\mathcal{N}^{k}.
\end{equation*}
As by hypothesis $\mathcal{N}^{k}\subseteq\mathcal{M}^{k}$ is generic, by the Baire Category Theorem it is dense and therefore each $\mathcal{N}^{k,l}\subseteq\mathcal{M}^{k}$ is dense and also open (because the $C^{k}$ topology is finer than the $C^{3}$ topology for every $k\geq 3$). Define
\begin{equation*}
    \mathcal{N}^{\infty,l}=\mathcal{N}^{3,l}\cap\mathcal{M}^{\infty}=\mathcal{N}^{3,l}\cap\bigcap_{k\geq 3}\mathcal{M}^{k}=\bigcap_{k\geq 3}\mathcal{N}^{k,l}.
\end{equation*}
Let us show that $\mathcal{N}^{\infty,l}\subseteq\mathcal{M}^{\infty}$ is dense. Pick $g_{0}\in\mathcal{M}^{\infty}$ and an open neighborhood $W$ of $g_{0}$ in $\mathcal{M}^{\infty}$. Let $k\in\mathbb{N}_{\geq 3}$ and $\delta>0$ be such that $\{g\in\mathcal{M}^{\infty}:d_{k}(g,g_{0})<\delta\}\subseteq W$ where $d_{k}$ is a metric which induces the $C^{k}$ topology on $\mathcal{M}^{k}$. By density of $\mathcal{N}^{k,l}$ in $\mathcal{M}^{k}$, there exists $g_{1}\in\mathcal{N}^{k,l}$ such that $d_{k}(g_{1},g_{0})<\frac{\delta}{2}$. On the other hand, as $\mathcal{M}^{\infty}\subseteq\mathcal{M}^{k}$ is dense with the $C^{k}$ topology and $\mathcal{N}^{k,l}\subseteq\mathcal{M}^{k}$ is open, there exists $g_{2}\in\mathcal{M}^{\infty}\cap\{g\in\mathcal{M}^{k}:d_{k}(g,g_{1})<\frac{\delta}{2}\}\cap\mathcal{N}^{k,l}$. Therefore by triangle inequality $g_{2}\in\{g\in\mathcal{M}^{\infty}:d_{k}(g,g_{0})<\delta\}\cap\mathcal{N}^{\infty,l}\subseteq W\cap\mathcal{N}^{\infty,l}$ so $\mathcal{N}^{\infty,l}\subseteq\mathcal{M}^{\infty}$ is dense. It is also open as the $C^{\infty}$ topology is finer than the $C^{3}$ one. Additionally,
\begin{equation*}
    \bigcap_{l\in\mathbb{N}}\mathcal{N}^{\infty,l}=(\bigcap_{l\in\mathbb{N}}\mathcal{N}^{3,l})\cap\mathcal{M}^{\infty}=\mathcal{N}^{3}\cap(\bigcap_{k\geq 3}\mathcal{M}^{k})=\bigcap_{k\geq 3}\mathcal{N}^{k}=\mathcal{N}^{\infty}.
\end{equation*}
This means that $\mathcal{N}^{\infty}\subseteq\mathcal{M}^{\infty}$ is generic with respect to the $C^{\infty}$ topology, as desired.
\end{proof}

\begin{proof}[Proof of Theorem \ref{Cinfty thm}]
For each $k\in\mathbb{N}_{\geq 3},$ define $\mathcal{N}^{k}\subseteq\mathcal{M}^{k}$ as the set of $C^{k}$ bumpy metrics. By Theorem \ref{bumpythm} in the case $k<\infty$ (which was already proved), $\mathcal{N}^{k}\subseteq\mathcal{M}^{k}$ is generic with respect to the $C^{k}$ topology for every $k\in\mathbb{N}_{\geq 3}$ and it clearly holds that $\mathcal{N}^{k'}=\mathcal{N}^{k}\cap\mathcal{M}^{k'}$ whenever $k'\geq k$. Therefore we can apply the lemma and deduce that $\mathcal{N}^{\infty}=\cap_{k\in\mathbb{N}}\mathcal{N}^{k}$ is generic in the space $\mathcal{M}^{\infty}$ of $C^{\infty}$ metrics. As $\mathcal{N}^{\infty}$ is precisely the set of $C^{\infty}$ bumpy metrics, this completes the proof of the theorem.
\end{proof}



\bibliography{main}

\providecommand{\bysame}{\leavevmode\hbox to3em{\hrulefill}\thinspace}
\providecommand{\MR}{\relax\ifhmode\unskip\space\fi MR }
\providecommand{\MRhref}[2]{%
  \href{http://www.ams.org/mathscinet-getitem?mr=#1}{#2}
}
\providecommand{\href}[2]{#2}
\begin{thebibliography}{1}

\bibitem{Chodosh}
Otis Chodosh and Christos Mantoulidis, \emph{The p-widths of a surface},
  preprint (2021), \url{https://arxiv.org/abs/2107.11684}.

\bibitem{Hirsch}
Morris~W. Hirsch, \emph{Differential topology}, Springer-Verlag New York Inc.,
  1976.

\bibitem{Inci}
Hasan Inci, Kappeler Thomas, and Topalov Peter, \emph{On the regularity of the
  composition of diffeomorphisms}, Memoirs of the American Mathematical Society
  \textbf{1062} (2013).

\bibitem{LiSta}
Xinze Li and Bruno Staffa, \emph{On the equidistribution of closed geodesics
  and geodesic nets}, preprint (2022), \url{https://arxiv.org/abs/2205.13694}.

\bibitem{Liokumovich}
Yevgeny Liokumovich and Bruno Staffa, \emph{Generic density of geodesic nets},
  preprint (2021), \url{https://arxiv.org/abs/2107.12340}.

\bibitem{Tsoy}
Tsoy-Wo Ma, \emph{Banach-{Hilbert} spaces, vector measures and group
  representations}, World Scientific Pub Co Inc, 2002.

\bibitem{Smale}
Stephen Smale, \emph{An infinite dimensional version of {Sard}'s theorem},
  Amer. J. Math. \textbf{87} (1965), 861--868.

\bibitem{White}
Brian White, \emph{The space of minimal submanifolds for varying {Riemannian}
  metrics}, Indiana University Mathematics Department \textbf{40} (1991),
  no.~1, 161--200.

\bibitem{White2}
\bysame, \emph{On the bumpy metrics theorem for minimal submanifolds}, Amer. J.
  Math. \textbf{139} (2017), no.~4, 1149--1155.

\end{thebibliography}
\bibliographystyle{amsplain}
\end{document}